\documentclass[12pt]{article}
\usepackage{latexsym}
\usepackage{epsfig}
\usepackage{amsthm,amsfonts,amssymb,amsmath}
\usepackage{stmaryrd}
\usepackage[latin1]{inputenc}
\usepackage[T1]{fontenc}

\date{}
\title{Hausdorff, Large Deviation and Legendre \\Multifractal Spectra of L\'evy Multistable Processes}
\author{R. Le Gu\'evel\\
\small{{\em Universit\'{e} de Rennes 2 - Haute Bretagne, Equipe de Statistique Irmar, UMR CNRS 6625}}\\
\small{{\em Place du Recteur Henri Le Moal, CS 24307, 35043 RENNES Cedex, France}}\\
\small{{\em ronan.leguevel@univ-rennes2.fr}}\\
\small{ and } \\
J. L\'{e}vy V\'{e}hel \\
\small{{\em Regularity team, INRIA Saclay and MAS laboratory}}\\
\small{{\em Ecole Centrale Paris, Grande Voie des Vignes, 92290, Chatenay Malabry, France}}\\
\small{{\em jacques.levy-vehel@inria.fr}}}
\def\bbbr{{\bf R}} % Real numbers
\def\bbbn{{\bf N}} % Natural numbers

\newtheorem{theo}{Theorem}

\newtheorem{note}{Note}
\newtheorem{lem}[theo]{Lemma}

\newcommand\Nb{N_{n}^{\varepsilon}(\beta)}

\newcommand\E{\mbox{\sf E}}
\renewcommand\P{{\sf P}}

\newcommand\h{h_Y(t)}

\newcommand{\one}{\ifmmode {\sf 1}\hspace{-.26em}{\sf
l}\hspace{-.35em}{\sf \_} \else ${\sf 1}\hspace{-.26em}{\sf
l}\hspace{-.35em}{\sf \_}$ \fi}

\newcommand\ho{H\"{o}lder }

\renewcommand{\Box}{\mbox{\rule{1ex}{1ex}}}

\begin{document}
\maketitle

\begin{abstract}

We compute the Hausdorff multifractal spectrum of two versions of multistable L\'evy motions. These processes
extend classical Lévy motion by letting the stability exponent $\alpha$ evolve in time. The spectra provide a
decomposition of $[0,1]$ into an uncountable disjoint union of sets with Hausdorff dimension one. We also compute
the increments-based large deviations multifractal spectrum of the independent increments multistable L\'evy motion.
This spectrum turns out to be concave and thus coincides with the Legendre multifractal spectrum, but it is different
from the Hausdorff multifractal spectrum. The independent increments multistable L\'evy motion
thus provides an example where the strong multifractal formalism does not hold.
\end{abstract}

\section{Introduction and background}

Multifractal analysis gives a fairly complete description of the singularity structure of measures, functions 
or stochastic processes. Various versions of multifractal analysis exist, which include the determinations 
of the so-called Hausdorff, large deviation, and Legendre multifractal spectra \cite{LVV}. 
Multifractal analysis has been performed for various measures \cite{AP,BMP}, functions \cite{Jaff1}, 
and stochastic processes \cite{Bar,BLV,AD,DJ,JLP}.
In the case of Lévy processes, substantially finer results have been obtained in \cite{PB} using 2-microlocal analysis.

This article deals with the multifractal analysis of extensions of Lévy stable motions known as {\it multistable
Lévy motions}. Generally speaking, multistable processes extend the well-known stable processes (see, {\it e.g.} \cite{Bk_Sam}) 
by letting the stability index $\alpha$ evolve in ``time''. These processes have been introduced in
\cite{KFJLV} and have been studied for instance in \cite{AA,HBCL,KFLL,LL2,LGLVL,IMKR}. 
They provide useful models in various applications where 
the data display jumps with varying intensity, such as financial records, EEG or natural terrains:
indeed, multistability is one practical way to deal with (increments-) non-stationarities observed in various real-world
phenomena, since a multistable process $X$ is tangent, at each time $u$, to a stable process $Z_u$ 
in the following sense \cite{F,F2}:
\begin{equation}
\lim_{r \to 0}\frac{X(u+rt) -X(u)}{r^{h}} = Z_{u}(t)
\label{locform1}
\end{equation}
for a suitable $h$ (the limit (\ref{locform1}) is taken either in finite dimensional distributions 
or, when $X$ has a version with càdlàg paths, in distribution - one then speaks of strong localisability).

Without loss of generality, we shall consider our processes on $[0,1]$. We will need the following ingredients:
\begin{itemize}
	\item $\alpha: [0,1] \to (1,2)$ is a $\mathcal{C}^1$ function.
	\item $(\Gamma_i)_{i \geq 1}$ is a sequence of arrival times of a Poisson process with unit arrival time.
	\item $(V_i)_{i \geq 1}$ is a sequence of i.i.d. random variables with uniform distribution on $[0,1]$.
	\item $(\gamma_i)_{i \geq 1}$ is a sequence of i.i.d. random variables with distribution
	$P(\gamma_i=1)=P(\gamma_i=-1)=1/2$.
\end{itemize}
The three sequences $(\Gamma_i)_{i \geq 1}$, $(V_i)_{i \geq 1}$, and $(\gamma_i)_{i \geq 1}$ are independent.
We denote $c=\inf_{u \in [0,1]}\limits \alpha(u)$, $d=\sup_{u \in [0,1]}\limits \alpha(u)$,

\noindent We shall consider two versions of Lévy multistable processes: the first one has 
independent but non stationary increments. It admits the following representation:
\begin{equation}\label{LII}
B(t) = \sum_{i=1}^{+\infty}\limits \gamma_i C_{\alpha(V_i)}^{1/\alpha(V_i)} \Gamma_i^{-1/ \alpha(V_i)} \mathbf{1}_{(V_i \leq t)},
\end{equation}
while the second one has correlated non stationary increments and reads:
\begin{equation}\label{LFB}
D(t)= C_{\alpha(t)}^{1/\alpha(t)}\sum_{i=1}^{\infty} \gamma_i \Gamma_i^{-1/\alpha(t)} \mathbf{1}_{[0,t]}(V_i),
\end{equation}
where  $C_u = \left( \int_{0}^{\infty} x^{-u} \sin x \ dx \right)^{-1}$.
Both processes are semi-martingales and are tangent, at each time $t$, to $\alpha(t)-$stable Lévy motion. 
See \cite{LGLVL} and the references therein for more details on these processes.

We shall also denote: 
 \begin{displaymath}
 Y(t)= \sum_{i=1}^{\infty} \gamma_i \Gamma_i^{-1/\alpha(t)} \mathbf{1}_{[0,t]}(V_i).
 \end{displaymath}

\section{Hausdorff multifractal spectra}
Let $\h$ denote the pointwise H\"{o}lder exponent of $Y$ at $t$. The Hausdorff multifractal analysis of $Y$ consists in
measuring the Hausdorff dimension (denoted $\dim_{H}$) of the sets $F_h = \{t \in [0,1]: \h =h \}$. 
The Hausdorff multifractal spectrum is the function $h \mapsto f_H(h) := \dim_{H} F_h$.

\noindent We will use the following notations:  $S=\cup_i \{V_i\}$, $\mathcal{S} = S^{\mathbf{N}}$ and 
 $\mathcal{R}_t = \{ (r_n)_{n \in \mathbf{N}} \in \mathcal{S}: r_n \rightarrow t\}$. 
 If $(r_n)_n \in \mathcal{R}_t$, we put $V_{\phi(n)}=r_n$.  
 Finally, define the positive function $\delta$ for $t \notin S$,
 
 \begin{displaymath}
 \delta(t) = \inf_{(V_{\phi(n)})_n \in \mathcal{R}_t}\limits \liminf_{i \to \infty}\limits -\frac{\log \phi(i)}{\log |V_{\phi(i)} -t|}.
  \end{displaymath}

\subsection{Main result}
The Hausdorff multifractal spectra of both $B$ and $D$ are described by the following theorem:
\begin{theo}\label{spectre}
With probability one, the common Hausdorff multifractal spectrum $f_H$ of $B$ and $D$ satisfies:
\begin{equation}\label{FH}
f_H(h)=
\begin{cases}
-\infty &\text{for $h<0$};\\
hd &\text{for $h\in [0,\frac{1}{d}]$};\\
1 &\text{for $h\in (\frac{1}{d},\frac{1}{c})$};\\
\dim_{H}\left( \{ t \in [0,1] : \alpha(t) =c \} \right) &\text{for $h=\frac{1}{c}$};\\
-\infty & \text{for $h>\frac{1}{c}$}.
\end{cases}
\end{equation} 
\end{theo}

\noindent Theorem \ref{spectre} follows from a series of lemmas that are proven in the next section:
\begin{lem}\label{lem1}
Almost surely, $t \mapsto Y(t)$ is càdlàg. 
\end{lem}

\begin{lem}\label{lem2}
Almost surely, $\forall t \in [0,1]\backslash S$, $\delta(t) \leq 1$. 
\end{lem}

\begin{lem}\label{lem3}
Almost surely, $\forall t \in [0,1]\backslash S$, $\h \leq \frac{\delta(t)}{\alpha(t)}$. 
\end{lem} 

\begin{lem}\label{lem4}
Let $g : [0,1] \rightarrow \bbbr$ be a càdlàg function, and $f$ be the function defined on $[0,1]$ 
by $f(t) = \int_0^t\limits g(x) dx$.
The pointwise H\"{o}lder exponent $h_f$ of $f$ verifies: $\forall t \in (0,1)$,
$$h_f(t) \geq 1.$$
\end{lem} 

\begin{lem}\label{lem5}
Almost surely, $\forall t \in [0,1]\backslash S$, $\h \geq \frac{\delta(t)}{\alpha(t)}$. 
\end{lem} 

\begin{lem}\label{lem6}
Almost surely, $\forall h <0, f_H(h)= -\infty$. 
\end{lem} 

\begin{lem}\label{lem7}
Almost surely, $ f_H(0)= 0$. 
\end{lem} 

\begin{lem}\label{lem8}
Almost surely, $\forall h \in (0, \frac{1}{d}], f_H(h)= hd$. 
\end{lem} 

\begin{lem}\label{lem9}
Almost surely, $\forall h \in (\frac{1}{d}, \frac{1}{c}), f_H(h)= 1$. 
\end{lem} 

\begin{lem}\label{lem10}
Almost surely, $f_H(\frac{1}{c})= \dim_{H}\left( \{ t \in [0,1] : \alpha(t) =c \} \right)$. 
\end{lem} 

\begin{lem}\label{lem11}
Almost surely, $\forall h > \frac{1}{c}, f_H(h)= -\infty$. 
\end{lem}

\subsection{Proofs of the lemmas}

\noindent  {\bf Proof of Lemma \ref{lem1}:}

 Set $Y_N(t) = \sum_{i=1}^{N}\limits \gamma_i \Gamma_i^{-1/\alpha(t)} \mathbf{1}_{[0,t]}(V_i).$
 Lemma 8 in \cite{LGLVL} states that, amost surely, $Y_N$ converges to $Y(t)$ uniformly on $[0,1]$.
 
 Fix $\varepsilon >0$ and choose $N_0 \in \bbbn$ such that, $\forall N \geq N_0$, 
 \begin{equation}\label{inegalunif}
  \sup_{t \in [0,1]}\limits |Y_N(t) - Y(t) | \leq \frac{\varepsilon}{3}.
 \end{equation}

 {\underline{$1^{\textrm{st}}$ case :} $t \in S$.
 
 Let $i_0 \in \bbbn$ be such that $t=V_{i_0}$, and $N_1=\max (i_0,N_0)$.
 Then, for $h \in \bbbr$,
 $$Y(t) - Y(t+h) =  Y(t) - Y_{N_1}(t) + Y_{N_1}(t) - Y_{N_1}(t+h) + Y_{N_1}(t+h) - Y(t+h).$$
 
 Since $\lim_{h \rightarrow 0^+}\limits Y_{N_1}(t) - Y_{N_1}(t+h)=0$, there exists $h_0 >0$ such that $\forall h \in (0,h_0)$,
 $$ |Y_{N_1}(t) - Y_{N_1}(t+h)| \leq \frac{\varepsilon}{3}.$$
 
As a consequence, $\forall h \in (0,h_0)$, $ |Y(t) - Y(t+h)| \leq \varepsilon$ and thus $\lim_{h \rightarrow 0^+}\limits Y(t) - Y(t+h)=0.$
 
$\lim_{h \rightarrow 0^-}\limits Y_{N_1}(t) - Y_{N_1}(t+h)= \lim_{h \rightarrow 0^-}\limits \sum_{i=1}^{N_1}\limits \gamma_i \Gamma_i^{-1/ \alpha(t)}\mathbf{1}_{(t+h,t]}(V_i) = \gamma_{i_0} \Gamma_{i_0}^{-1/ \alpha(V_{i_0})},$
thus $$\lim_{h \rightarrow 0^-}\limits Y_{N_1}(t) - Y_{N_1}(t+h) - \gamma_{i_0} \Gamma_{i_0}^{-1/ \alpha(V_{i_0})} =0.$$
 
 Choose $h_0 <0$ such that $\forall h \in (h_0,0)$,
 $$ |Y(t) - Y(t+h) - \gamma_{i_0} \Gamma_{i_0}^{-1/ \alpha(V_{i_0})}| \leq \varepsilon.$$
 
 Thus, $$\lim_{h \rightarrow 0^-}\limits Y(t) - Y(t+h) = \gamma_{i_0} \Gamma_{i_0}^{-1/ \alpha(V_{i_0})}.$$
 
 {\underline{$2^{\textrm{nd}}$ case :} $t \notin S$
 
 Since $\lim_{h \rightarrow 0}\limits |Y_{N_0}(t+h) - Y_{N_0}(t)|=0$, there exists  $h_0 >0$ such that $\forall |h| < h_0$,
$$|Y_{N_0}(t) - Y_{N_0}(t+h)| \leq \frac{\varepsilon}{3}.$$ 
Using (\ref{inegalunif}), one thus has, for $|h| < h_0$,  
 $|Y(t) - Y(t+h)| \leq \varepsilon$, and thus $\lim_{h \rightarrow 0}\limits |Y(t+h) - Y(t)|=0$ \Box

 \begin{note}\label{jumps}
We have shown precisely that $Y$ is càdlàg with set of jump points exactly equal to $S$. The jump at point $V_i$ is of size $\gamma_i \Gamma_i^{-1/ \alpha(V_i)}$.
  \end{note}

\bigskip

\noindent  {\bf Proof of Lemma \ref{lem2}:}

 For $j\geq 1$, $k=1,...,2^j$, let $E_{k,j}$ denote the interval
 $$E_{k,j} = \left[ \frac{k-1}{2^j} - \frac{1}{2^{1+(j+1)(1-\frac{1}{\sqrt{j}})}} , \frac{k-1}{2^j} + \frac{1}{2^{1+(j+1)(1-\frac{1}{\sqrt{j}})}}\right).$$
 
 Let us show that $\liminf_{j \to \infty}\limits \bigcap_{k=1}^{2^j}\limits \bigcup_{i=2^j}^{2^{j+1}-1}\limits \{ V_i \in E_{k,j} \}  \subset \{ \forall t \notin S, \delta(t) \leq 1 \}$ first, and then 
 that $\P \left( \liminf_{j \to \infty}\limits \bigcap_{k=1}^{2^j}\limits \bigcup_{i=2^j}^{2^{j+1}-1}\limits \{ V_i \in E_{k,j} \} \right) =1.$
 We denote $a_j = \frac{1}{2^{1+(j+1)(1-\frac{1}{\sqrt{j}})}}.$
 
 Assume that there exists $J_0 \in \bbbn$ such that for all $j \geq J_0$, and all $k=1,...,2^j$, we can fix $i(k,j) \in [2^j, 2^{j+1}-1 ]$ with $V_{i(k,j)} \in E_{k,j}$.
 Let $t \notin S$. For all $j \geq J_0$, there exists $k(j) \in [1,2^j]$ such that $t \in E_{k(j),j}$, because $2a_j \geq \frac{1}{2^j}$.
 
As a consequence, $$|t-V_{i(k(j),j)}| \leq \frac{1}{2^{(j+1)(1-\frac{1}{\sqrt{j}})}} \leq \frac{1}{i(k(j),j)^{1-\frac{1}{\sqrt{j}}}}.$$ Hence
 $ - \frac{\log i(k(j),j) }{\log |t-V_{i(k(j),j)}|} \leq \frac{1}{1-\frac{1}{\sqrt{j}}} .$ This entails 
 $\liminf_{j \to \infty}\limits  - \frac{\log i(k(j),j) }{\log |t-V_{i(k(j),j)}|} \leq 1$, and, since $V_{i(k(j),j)}$ tends to $t$, $\delta(t) \leq 1$.
 
 Finally, distinguishing the cases $k=1$, $k=2, \ldots, 2^j-1$ and $k=2^j$, one estimates
 
 \begin{eqnarray*}
  \P \left( \bigcup_{k=1}^{2^j}\limits \bigcap_{i=2^j}^{2^{j+1}-1}\limits \{ V_i \notin E_{k,j} \} \right) & \leq & \sum_{k=1}^{2^j}\limits \P \left( \bigcap_{i=2^j}^{2^{j+1}-1}\limits \{ V_i \notin E_{k,j} \} \right) \\
 & \leq & a_j + (1-( \frac{2^j -1}{2^j}- a_j) ) + \sum_{k=2}^{2^j -1}\limits \left( \P \left( \{ V_1 \notin E_{k,j} \} \right) \right)^{2^j} \\
 & \leq & 2 a_j + \frac{1}{2^j} + (2^j -2)(1-2a_j)^{2^j}.
 \end{eqnarray*}
Since $\sum_{j=1}^{+ \infty}\limits (2^j -2)(1-2a_j)^{2^j} < +\infty,$ Borel-Cantelli lemma allows us to conclude. \Box

\bigskip 

\noindent {\bf Proof of Lemma \ref{lem3}:}
 
Recall Note \ref{jumps}. Lemma 1 of \cite{Jaff1} entails that, for all sequences $V_{\phi(i)} \in \mathcal{R}_t$, 
 and all $t \notin S$, 
 \begin{eqnarray*}
  \h & \leq & \liminf_{i \rightarrow +\infty}\limits  \frac{\log | \Gamma_{\phi(i)}^{-\frac{1}{\alpha(V_{\phi(i)})}}|}{\log | V_{\phi(i)} -t|}\\
  & =& \liminf_{i \rightarrow +\infty}\limits -\frac{1}{\alpha(V_{\phi(i)})} \frac{\log | \Gamma_{\phi(i)} |}{\log | V_{\phi(i)} -t| }.
 \end{eqnarray*}
Since $\alpha$ is continuous, $\phi(i)$ tends to infinity, the sequences $(V_{\phi(i)})_i$ converges to $t$, 
and almost surely $(\frac{\Gamma_i}{i})_i$ tends to $1$  when $i$ tends to infinity, one obtains
$$\h \leq \frac{1}{\alpha(t)} \liminf_{i \rightarrow +\infty}\limits - \frac{\log | \phi(i) |}{\log | V_{\phi(i)} -t| }.$$

This inequality holds for all sequences $V_{\phi(i)} \in \mathcal{R}_t$, and thus,
$\forall t \notin S$, $\h \leq \frac{\delta(t)}{\alpha(t)}$ \Box

\bigskip 
 
\noindent {\bf Proof of Lemma \ref{lem4}:}

Since $g$ is càdlàg, $h_g$ is non negative for all $t$. Integration increases pointwise regularity by at least one, and thus
$h_f(t) \geq 1$ for all $t$. An alternative direct proof goes as follows:  
let $t \in (0,1)$, and $h >0$. One computes
 
 \begin{eqnarray*}
  \frac{f(t+h) - f(t)}{h} - g(t^+) & = & \frac{1}{h} \int_{t}^{t+h} (g(x) - g(t^+)) dx\\
  & = & \int_{0}^{1} (g(t+sh) - g(t^+) )ds.
 \end{eqnarray*}

 $\forall s \in (0,1)$, $\lim_{h \rightarrow 0^+}\limits (g(t+sh) - g(t^+)) = 0$. Since $g$ is càdlàg, it is bounded and thus: 
 $$\lim_{h \rightarrow 0^+}\limits \frac{f(t+h) - f(t)}{h} = g(t^+).$$
 
Likewise, for  $h<0$, 
 $$\frac{f(t+h) - f(t)}{h} - g(t^-) = \int_{0}^{1} (g(t+sh) - g(t^-) )ds,$$
 and $$\lim_{h \rightarrow 0^-}\limits \frac{f(t+h) - f(t)}{h} = g(t^-).$$
 This entails $h_f(t) \geq 1$ \Box
 
\bigskip
 
 \noindent {\bf Proof of Lemma \ref{lem5}:}

Theorem 7 of \cite{LGLVL} states that:
 $$ D(t) = A(t) + B(t),$$
 where 
 $$A(t) = \int_{0}^{t} \sum_{i=1}^{+\infty}\limits \gamma_i \frac{d\left( C_{\alpha(.)}^{1/\alpha(.)} \Gamma_i^{-1/ \alpha(.)}\right)}{dt} (s) \mathbf{1}_{[0,s)}(V_i) \ ds.$$

In addition, Lemma 8 of \cite{LGLVL} entails that $\left( \sum_{i=1}^{N}\limits \gamma_i \frac{d\left( C_{\alpha(.)}^{1/\alpha(.)} \Gamma_i^{-1/ \alpha(.)}\right)}{dt} (s) \mathbf{1}_{[0,s)}(V_i) \right)_N$ 
 converges to $ \sum_{i=1}^{+\infty}\limits \gamma_i \frac{d\left( C_{\alpha(.)}^{1/\alpha(.)} \Gamma_i^{-1/ \alpha(.)}\right)}{dt} (s) \mathbf{1}_{[0,s)}(V_i)$  uniformly on $[0,1]$ .
The same proof as the one of Lemma \ref{lem1} shows that $s \mapsto \sum_{i=1}^{+\infty}\limits \gamma_i \frac{d\left( C_{\alpha(.)}^{1/\alpha(.)} \Gamma_i^{-1/ \alpha(.)}\right)}{dt} (s) \mathbf{1}_{[0,s)}(V_i)$ is 
càdlàg. Lemma \ref{lem4} then entails that $h_A(t) \geq 1$. Since $c>1$, it is thus sufficient to show that
$h_B(t) \geq \frac{\delta(t)}{\alpha(t)}$.
 
Write $B(t) = W(t) + Z(t)$ where 
\begin{equation}\label{eq:W}
W(t) = \sum_{i=1}^{+\infty}\limits \gamma_i C_{\alpha(V_i)}^{1/\alpha(V_i)} \left( \Gamma_i^{-1/ \alpha(V_i)} - i^{-1/ \alpha(V_i)} \right) \mathbf{1}_{(V_i \leq t)}
\end{equation}
and
\begin{equation}\label{eq:Z}
Z(t) = \sum_{i=1}^{+\infty}\limits \gamma_i C_{\alpha(V_i)}^{1/\alpha(V_i)} i^{-1/ \alpha(V_i)} \mathbf{1}_{(V_i \leq t)}.
\end{equation}
 
Set $I_{k,m} = [\frac{k}{2^m}, \frac{k+1}{2^m})$, $N_{m,j,k} = \textrm{Card}\left\{ V_i, i=2^j,...,2^{j+1}-1, V_i \in I_{k,m} \right\}$, $d_{k,m} = \max_{u \in I_{k,m}}\limits \alpha(u)$
 and $C_0 = \max_{t \in [0,1]}\limits C_{\alpha(t)}^{1/ \alpha(t)}$. Define 
 $$M_t^{m,j,k} = \sum_{i=2^j}^{2^{j+1}-1}\limits \gamma_i C_{\alpha(V_i)}^{1/\alpha(V_i)} i^{-1/ \alpha(V_i)} \mathbf{1}_{[0,t]\cap I_{k,m}}(V_i).$$
It is easily seen that 
 
 \begin{equation}\label{inegalmart}
 \sup_{(s,t) \in I_{k,m}^2}\limits \left|\sum_{i=2^j}^{2^{j+1}-1}\limits \gamma_i C_{\alpha(V_i)}^{1/\alpha(V_i)} i^{-1/ \alpha(V_i)} \mathbf{1}_{[s,t)}(V_i) \right| \leq 2 \sup_{t \in [0,1]} \left| M_t^{m,j,k} \right|.
\end{equation}

For $t \in (0,1)$, let $\alpha_n(t) = \sqrt{n} \left[ \frac{1}{n} \sum_{i=1}^{n}\limits \mathbf{1}_{V_i \leq t} - t\right]$
denote the empirical process, and $w_n(a) = \sup_{|t-s| \leq a}\limits | \alpha_n(t) - \alpha_n(s)|$ denote 
the oscillation modulus of $\alpha_n$.
We apply Lemma 2.4 of \cite{Stu} with 
$$ m \geq 3, a=2^{-m}, s=m^{1/4} \sqrt{j}, n=2^j, \delta=\frac{1}{2}.$$
This yields that there exists $ M_0 \in \bbbn$ such that
$$\forall \ m \geq M_0, \ \exists \ j(m) \mbox{ with } m \leq j(m) \leq 2m \mbox{ such that } \forall \ j \geq j(m): \ m^{1/4} \sqrt{j} \leq \sqrt{2^{j-m}}$$
and
\begin{equation}\label{inegalStu}
 \P \left( \sup_{0 \leq |t-s| \leq \frac{1}{2^m}}\limits | \alpha_{2^j}(t) - \alpha_{2^j}(s)|> \frac{m^{1/4} \sqrt{j}}{\sqrt{2^m}}\right) \leq 256 \times 2^m e^{- \frac{j\sqrt{m}}{64}}.
\end{equation}
We need to estimate $N_{m,j,k}$ and $\sup_{(s,t) \in I_{k,m}^2}\limits \left|\sum_{i=2^j}^{2^{j+1}-1}\limits \gamma_i C_{\alpha(V_i)}^{1/\alpha(V_i)} i^{-1/ \alpha(V_i)} \mathbf{1}_{[s,t)}(V_i) \right|$
for $k=0,...,2^m-1$ and $j \geq m$.

\noindent $\bullet$ \underline{Study of $N_{m,j,k}$ for $m \leq j \leq j(m)-1$:}

Set $X_i = \mathbf{1}_{\frac{k}{2^m} \leq V_i \leq \frac{k+1}{2^m}}$ and $n=2^j$. $(X_i)_i$ is an i.i.d. sequence of
Bernoulli random variables with parameter $p=\frac{1}{2^m}$. For $m \leq j \leq j(m)-1$, one has
$\sqrt{m} j \geq 2^{j-m} =np$, and thus, for $a>0$, using a classical bound on the sum of i.i.d. Bernoulli random
variables,
 \begin{eqnarray*}
   \P \left( N_{m,j,k} > a \sqrt{m}j \right) &\leq &\P \left( \sum_{i=2^j}^{2^{j+1}-1}\limits X_i \geq a 2^{j-m} \right) \\
   & = & \P \left( \sum_{i=2^j}^{2^{j+1}-1}\limits X_i \geq a n p \right) \\
	 &\leq & \frac{1}{a^n}\\
     & \leq & \frac{1}{a^j}.
 \end{eqnarray*}
 
As a consequence,
  \begin{eqnarray*}
     \P \left( \bigcup_{j=m}^{j(m)-1}\limits \bigcup_{k=0}^{2^m-1}\limits \left\{ N_{m,j,k} > a \sqrt{m}j \right\} \right) & \leq & \sum_{j=m}^{2m}\limits \sum_{k=0}^{2^{m}-1}\limits \frac{1}{a^j}\\
     & \leq & 2^m \sum_{j=m}^{+\infty}\limits a^{-j}\\
     & \leq & \frac{a}{a-1} \left(\frac{2}{a}\right)^m.
  \end{eqnarray*}
 
Choosing $a=3$, Borel Cantelli lemma entails that
$$\P \left(\limsup_{m \rightarrow +\infty}\limits   \bigcup_{j=m}^{j(m)-1}\limits \bigcup_{k=0}^{2^m-1}\limits \left\{ N_{m,j,k} > 3 \sqrt{m}j \right\}\right) = 0.$$

\noindent $\bullet$ \underline{Study of $N_{m,j,k}$ for $j \geq j(m)$:}

\begin{eqnarray*}
     \P \left( N_{m,j,k} > 2.2^{j-m} \right) & = & \P \left( \sqrt{2^j} \left( \alpha_{2^j}\left(\frac{k+1}{2^m}\right) 
     - \alpha_{2^j}\left(\frac{k}{2^m}\right)\right) + 2^{j-m} \geq 2.2^{j-m} \right) \\
& \leq & \P \left( \sqrt{2^j} w_{2^j}(\frac{1}{2^m}) \geq 2^{j-m} \right)\\
& = & \P \left( w_{2^j}(\frac{1}{2^m}) \geq \frac{\sqrt{2^{j-m}}}{\sqrt{2^m}} \right)\\
& \leq & \P \left( w_{2^j}(\frac{1}{2^m}) \geq \frac{m^{1/4}\sqrt{j}}{\sqrt{2^m}} \right)\\
& \leq & 256 \times 2^m e^{- \frac{j\sqrt{m}}{64}}
\end{eqnarray*}
using (\ref{inegalStu}). As a consequence, 

\begin{eqnarray*}
\P \left( \bigcup_{j=j(m)}^{+\infty}\limits \bigcup_{k=0}^{2^m-1}\limits \left\{ N_{m,j,k} > 2 2^{j-m} \right\} \right) & \leq & \sum_{j=m}^{+\infty}\limits 2^m (256 \times 2^m e^{- \frac{j\sqrt{m}}{64}})\\
& \leq & 256. 4^m . \frac{e^{- \frac{m\sqrt{m}}{64}}}{1-e^{- \frac{\sqrt{m}}{64}}}.
\end{eqnarray*}
Borel Cantelli lemma yields 
$$\P \left(\limsup_{m \rightarrow +\infty}\limits   \bigcup_{j=j(m)}^{+\infty}\limits \bigcup_{k=0}^{2^m-1}\limits \left\{ N_{m,j,k} \geq 2. 2^{j-m}\right\}\right) = 0.$$

\noindent $\bullet$ \underline{Study of $\sup_{(s,t) \in I_{k,m}^2}\limits \left|\sum_{i=2^j}^{2^{j+1}-1}\limits \gamma_i C_{\alpha(V_i)}^{1/\alpha(V_i)} i^{-1/ \alpha(V_i)} \mathbf{1}_{[s,t)}(V_i) \right| $ for $m \leq j \leq j(m)-1$:}

Almost surely, there exists $m_0$ such that, for $m\geq m_0$ and for
$m \leq j \leq j(m)-1$, $\forall k=0,...,2^m-1$, $N_{m,j,k} \leq 3 \sqrt{m} j$, 
and $\forall (s,t) \in I_{k,m}^2$, $\forall i=2^j,...,2^{j+1}-1$,

$$ \left| \gamma_i C_{\alpha(V_i)}^{1/\alpha(V_i)} i^{-1/ \alpha(V_i)} \mathbf{1}_{[s,t)}(V_i) \right| \leq \frac{C_0}{2^{\frac{j}{d_{k,m}}}}$$
thus 

$$\sup_{(s,t) \in I_{k,m}^2}\limits \left|\sum_{i=2^j}^{2^{j+1}-1}\limits \gamma_i C_{\alpha(V_i)}^{1/\alpha(V_i)} i^{-1/ \alpha(V_i)} \mathbf{1}_{[s,t)}(V_i) \right| \leq  \frac{ 3 \sqrt{m} j C_0}{2^{\frac{j}{d_{k,m}}}}.$$

\noindent $\bullet$ \underline{Study of $\sup_{(s,t) \in I_{k,m}^2}\limits \left|\sum_{i=2^j}^{2^{j+1}-1}\limits \gamma_i C_{\alpha(V_i)}^{1/\alpha(V_i)} i^{-1/ \alpha(V_i)} \mathbf{1}_{[s,t)}(V_i) \right| $ for $j \geq j(m)$:}

We consider the events $$E_{m,j,k} = \left\{ N_{m,j,k} \leq 2.2^{j-m} \right\},$$ 
$$F_{m,j,k} = \left\{ \sup_{(s,t) \in I_{k,m}^2}\limits \left|\sum_{i=2^j}^{2^{j+1}-1}\limits \gamma_i C_{\alpha(V_i)}^{1/\alpha(V_i)} i^{-1/ \alpha(V_i)} \mathbf{1}_{[s,t)}(V_i) \right| \geq \frac{j C_0 \sqrt{2.2^{j-m}}}{2^{\frac{j}{d_{k,m}}}} \right\}$$ 
and
$$G_{m,j,k} = \left\{ \sup_{t \in [0,1]}\limits \left|M_t^{m,j,k} \right| \geq \frac{j C_0 \sqrt{2.2^{j-m}}}{2.2^{\frac{j}{d_{k,m}}}} \right\}.$$ 

\noindent Relation (\ref{inegalmart}) entails that $\P \left( F_{m,j,k} \cap E_{m,j,k}\right) \leq \P \left( G_{m,j,k} \cap E_{m,j,k}\right).$

\noindent In the following computation, $l$ corresponds to the number of terms $V_i$ belonging to $I_{k,m}$, and $n$ 
corresponds to the number of those $V_i$ among them that contribute to the supremum of $G_{m,j,k}$.

$$G_{m,j,k} \cap E_{m,j,k} \subset \bigcup_{l=1}^{2.2^{j-m}}\limits \bigcup_{l_1,...,l_l \in [2^j,2^{j+1}-1]}\limits \left[ G_{m,j,k} \cap \left( \bigcap_{i \in \{l_1,...l_l \}}\limits  V_i \in I_{k,m}\right) \cap \left(\bigcap_{i \notin \{l_1,...l_l \}}\limits  V_i \notin I_{k,m} \right)\right].$$

Using independence of the $V_i$, 

$$\P \left( G_{m,j,k} \cap E_{m,j,k} \right) \leq \sum_{l=1}^{2.2^{j-m}}\limits \sum_{l_1,...,l_l \in [2^j,2^{j+1}-1]}\limits \P \left(  G_{m,j,k} \cap  \bigcap_{i \in \{l_1,...l_l \}}\limits  V_i \in I_{k,m} \right) \P \left(  \bigcap_{i \notin \{l_1,...l_l \}}\limits  V_i \notin I_{k,m} \right).$$

Now, $\P \left(  \bigcap_{i \notin \{l_1,...l_l \}}\limits  V_i \notin I_{k,m} \right) = \left( 1- \frac{1}{2^m}\right)^{2^j-l}$.
Let us fix an order on the $V_i$ belonging to $I_{k,m}$: there are $l!$ possibilities, all equiprobable, and thus 
$$\P \left(  G_{m,j,k} \cap  \bigcap_{i \in \{l_1,...l_l \}}\limits  V_i \in I_{k,m} \right) = (l!) \P \left( G_{m,j,k} \cap ( \frac{k}{2^m} < V_{l_1} < ... < V_{l_l} < \frac{k+1}{2^m})\right).$$

Let $A_{k,m}^l$ denote the event $\{  \frac{k}{2^m} < V_{l_1} < ... < V_{l_l} < \frac{k+1}{2^m} \}$. Then

\begin{eqnarray*} 
\P \left(  G_{m,j,k} \cap A_{k,m}^l \right) &\leq& \sum_{n=1}^{l}\limits \P \left( A_{k,m}^l \cap \left| \sum_{i=1}^{n}\limits \gamma_{l_i} C_{\alpha(V_{l_i})}^{1/\alpha(V_{l_i})} l_i^{-1/ \alpha(V_{l_i})} \right| \geq  \frac{j C_0 \sqrt{2.2^{j-m}}}{2.2^{\frac{j}{d_{k,m}}}}\right) \\
& \leq & \sum_{n=1}^{l}\limits  \int_{ \frac{k}{2^m} < x_1 < ... < x_l < \frac{k+1}{2^m} }\limits \P \left( \left| \sum_{i=1}^{n}\limits \gamma_{l_i} C_{\alpha(x_i)}^{1/\alpha(x_i)} l_i^{-1/ \alpha(x_i)} \right| \geq  \frac{j C_0 \sqrt{n}}{2.2^{\frac{j}{d_{k,m}}}} \right) dx_1 ... dx_l    .
\end{eqnarray*}
The probability inside the integral above may be estimated with the help of Lemma 1.5 in \cite{LTal}:

$$ \P \left( \left| \sum_{i=1}^{n}\limits \gamma_{l_i} \frac{C_{\alpha(x_i)}^{1/\alpha(x_i)}}{C_0} \frac{2^{\frac{j}{d_{k,m}}}}{l_i^{1/ \alpha(x_i)}} \right| \geq  \frac{j \sqrt{n}}{2} \right) \leq 2e^{-\frac{j^2}{8}}.$$

One then computes
\begin{eqnarray*} 
 \P \left( G_{m,j,k} \cap E_{m,j,k} \right) &\leq &\sum_{l=1}^{2.2^{j-m}}\limits \sum_{l_1,...,l_l \in [2^j,2^{j+1}-1]}\limits (1-\frac{1}{2^m})^{2^j-l} (l!) \sum_{n=1}^{l}\limits 2e^{-\frac{j^2}{8}} \int_{ \frac{k}{2^m} < x_1 < ... < x_l < \frac{k+1}{2^m} }\limits dx_1 ... dx_l \\
 & \leq & \sum_{l=1}^{2.2^{j-m}}\limits \frac{(2^j)!}{(l!)(2^j-l)!} (1-\frac{1}{2^m})^{2^j-l} (l!) 2l e^{-\frac{j^2}{8}} \frac{1}{2^{ml}} \frac{1}{l!} \\
 & \leq & 2 e^{-\frac{j^2}{8}} \sum_{l=1}^{2.2^{j-m}}\limits l \frac{(2^j)!}{(l!)(2^j-l)!} (1-\frac{1}{2^m})^{2^j-l} \frac{1}{2^{ml}} \\
 & \leq & 4. 2^{j-m} e^{-\frac{j^2}{8}} \sum_{l=1}^{2.2^{j-m}}\limits \frac{(2^j)!}{(l!)(2^j-l)!} (1-\frac{1}{2^m})^{2^j-l} \frac{1}{2^{ml}} \\
 & \leq & 4. 2^{j-m} e^{-\frac{j^2}{8}}.
\end{eqnarray*}
As a consequence, 
\begin{eqnarray*} 
 \P \left( \bigcup_{j=j(m)}^{+\infty}\limits \bigcup_{k=0}^{2^m-1}\limits \left( F_{m,j,k} \cap E_{m,j,k}\right) \right) & \leq & 4 \sum_{j=m}^{+\infty}\limits 2^j e^{- \frac{j^2}{8}}\\
 & \leq & \sum_{j=m}^{+\infty}\limits e^{-j}
 \end{eqnarray*}
for $m \geq 100$. Borel Cantelli lemma entails that

$$\P \left(\limsup_{m \rightarrow +\infty}\limits \bigcup_{j=j(m)}^{+\infty}\limits \bigcup_{k=0}^{2^m-1}\limits F_{m,j,k} \right) =0.$$

\noindent $\bullet$ \underline{Computation of the \ho exponent:}

Let $t \in (0,1)$, $t \notin S$ and let $U$ be an open interval of $(0,1)$ containing $t$. 
Denote $d_U = \max_{t \in U}\limits \alpha(t)$. If $\delta(t) = 0$, then $\h=0$ and the formula holds. 
Suppose now $\delta(t) >0$. Let $\varepsilon >0$ be such that $\delta(t) > \varepsilon$ 
and $\frac{1}{d_U}- \frac{1}{2} > \varepsilon$. There exists $i_0 \in \bbbn$ such that $\forall i \geq i_0$, $|t-V_i| \geq \frac{1}{i^{\frac{1}{\delta(t) - \varepsilon}}}$. 
Choose $m$ large enough so that $\frac{1}{2^m} < \min \{ |t-V_i|, i=1,...,i_0 \}$.
Let $j_0 = [ m(\delta(t)- \varepsilon)]$. Increasing $m$ if necessary, we may and will assume that
$i_0 \leq 2^{j_0}$, and $\forall j \geq j_0$, $j \leq 2^{j \varepsilon}$. 
Let $s \in (0,1)$ be such that $\frac{1}{2^{m+2}} \leq |t-s| < \frac{1}{2^{m+1}}$. There exists
$k \in \{0,...,2^m-1\}$ such that $(t,s) \in I_{k,m}^2$. Increasing again $m$ if necessary, we may 
assume that $I_{k,m} \subset U$. 
Then $d_{k,m} \leq d_U$, and, for $i \leq 2^{j_0}$, $\mathbf{1}_{[s,t)}(V_i) =0$. One computes:

\begin{eqnarray*}
 \left| Z(t)- Z(s) \right| & = & \left| \sum_{i=1}^{+\infty}\limits \gamma_i C_{\alpha(V_i)}^{1/\alpha(V_i)} i^{-1/ \alpha(V_i)} \mathbf{1}_{[s,t)}(V_i)\right| \\
 & = & \left| \sum_{j=j_0}^{+\infty}\limits \sum_{i=2^j}^{2^{j+1}-1}\limits \gamma_i C_{\alpha(V_i)}^{1/\alpha(V_i)} i^{-1/ \alpha(V_i)} \mathbf{1}_{[s,t)}(V_i)\right| \\
 & \leq & \sum_{j=j_0}^{j(m)}\limits \left| \sum_{i=2^j}^{2^{j+1}-1}\limits \gamma_i C_{\alpha(V_i)}^{1/\alpha(V_i)} i^{-1/ \alpha(V_i)} \mathbf{1}_{[s,t)}(V_i)\right| + \sum_{j=j(m)}^{+\infty}\limits \left|\sum_{i=2^j}^{2^{j+1}-1}\limits \gamma_i C_{\alpha(V_i)}^{1/\alpha(V_i)} i^{-1/ \alpha(V_i)} \mathbf{1}_{[s,t)}(V_i)\right|\\
 & \leq & \sum_{j=j_0}^{j(m)}\limits \frac{3 \sqrt{m} j }{2^{\frac{j}{d_{k,m}}}}+ \sum_{j=j(m)}^{+\infty}\limits \frac{j C_0 \sqrt{2.2^{j-m}} }{2^{\frac{j}{d_{k,m}}}}\\
 & \leq & 2(3+\sqrt{2} C_0) \sqrt{m} \sum_{j=j_0}^{+\infty}\limits \frac{j \sqrt{2^{j-m}} }{2^{\frac{j}{d_{k,m}}}}\\
 & \leq & \frac{6(1+C_0) \sqrt{m}}{\sqrt{2^m}} \sum_{j=j_0}^{+\infty}\limits 2^{j(\varepsilon + \frac{1}{2} - \frac{1}{d_U})}\\
 & \leq & \frac{6(1+C_0) \sqrt{m}}{\sqrt{2^m}}  \frac{2^{j_0(\varepsilon + \frac{1}{2} - \frac{1}{d_U})}}{1- 2^{\varepsilon + \frac{1}{2} - \frac{1}{d_U}}}\\
 & \leq & K_{U,\varepsilon,C_0} \sqrt{m} 2^{m \left[ -\frac{1}{2} + (\delta(t) - \varepsilon)(\varepsilon + \frac{1}{2} - \frac{1}{d_U}) \right]}\\
& \leq & K_{U,\varepsilon,C_0} \sqrt{| \log |t-s||} |t-s|^{(\delta(t) - \varepsilon)(\frac{1}{d_U} - \varepsilon) + \frac{1}{2} - \frac{1}{2}(\delta(t) - \varepsilon)}.
 \end{eqnarray*}
Since $\frac{1}{2} - \frac{1}{2}(\delta(t) - \varepsilon) > 0$, 
$$\left| Z(t)- Z(s) \right| \leq K_{U,\varepsilon,C_0} \sqrt{| \log |t-s||} |t-s|^{(\delta(t) - \varepsilon)(\frac{1}{d_U} - \varepsilon) }.$$

Let us now study $W$ (recall \eqref{eq:W}): there exists $i_0$ such that, for all $i \geq i_0$, 
\begin{eqnarray*}
 \left| \Gamma_i^{-1/\alpha(V_i)} - i^{-1/\alpha(V_i)} \right| & \leq & \frac{K_{c,d}}{i^{\frac{1}{d_{k,m}}+\frac{1}{2}}}\\
 & \leq & \frac{K_{c,d}}{i^{\frac{1}{d_U}+\frac{1}{2}}}.
\end{eqnarray*}

One then computes: 

\begin{eqnarray*}
 \left| W(t)- W(s) \right| & = & \left| \sum_{j=j_0}^{+\infty}\limits \sum_{i=2^j}^{2^{j+1}-1}\limits \gamma_i C_{\alpha(V_i)}^{1/\alpha(V_i)} ( \Gamma_i^{-1/ \alpha(V_i)} - i^{-1/ \alpha(V_i)}) \mathbf{1}_{[s,t)}(V_i)\right| \\
& \leq & \sum_{j=j_0}^{j(m)}\limits \frac{3 K_{c,d} C_0 \sqrt{m} j }{2^{j(\frac{1}{2} + \frac{1}{d_U} )}}+ \sum_{j=j(m)}^{+\infty}\limits \frac{K_{c,d} C_0 2.2^{j-m} }{2^{j(\frac{1}{2} + \frac{1}{d_U} )}}\\
& \leq & K \sqrt{m} \sum_{j=j_0}^{+\infty}\limits 2^{j(\varepsilon - \frac{1}{2} - \frac{1}{d_U})} + \frac{K}{2^m} \sum_{j=j_0}^{+\infty}\limits 2^{j(\frac{1}{2} - \frac{1}{d_U})}\\
& \leq & K \sqrt{|\log |t-s| |} \frac{2^{m(\delta(t)-\varepsilon)(\varepsilon - \frac{1}{2} - \frac{1}{d_U})}}{1-2^{\varepsilon - \frac{1}{2} - \frac{1}{d_U}}} + K |t-s| \frac{2^{m(\delta(t) - \varepsilon)(\frac{1}{2} - \frac{1}{d_U})}}{1-2^{\frac{1}{2} - \frac{1}{d_U}}}\\
& \leq & K \sqrt{|\log |t-s| |} |t-s|^{(\delta(t)-\varepsilon)(\frac{1}{2} + \frac{1}{d_U}-\varepsilon)} + K |t-s|^{1+(\delta(t) - \varepsilon)( \frac{1}{d_U} - \frac{1}{2})}.
 \end{eqnarray*}
Gathering our results, we have shown that:

$$\left| B(t) - B(s) \right| \leq K \sqrt{|\log |t-s| |} |t-s|^{(\delta(t) - \varepsilon)(\frac{1}{d_U} - \varepsilon) }.$$
%thus  $\forall \varepsilon >0$ small enough,
%$$h_B(t) \geq (\delta(t) - \varepsilon)(\frac{1}{d_U} - \varepsilon),$$
In other words, $h_B(t) \geq \frac{\delta(t)}{d_U}$ for any open interval $U$ containing $t$. 
Letting the diameter of $U$ go to 0, one gets $h_B(t) \geq \frac{\delta(t)}{\alpha(t)}$. \Box

\bigskip

\noindent  {\bf Proof of Lemma \ref{lem6}:}

Lemma \ref{lem1} entails that $Y$ is almost surely a càdlàg process. Thus, for all $t \in (0,1)$, $\h \geq 0$. \Box

\bigskip

\noindent  {\bf Proof of Lemma \ref{lem7}:}

We seek to compute the Hausdorff dimension of $F_0 = \{ t \in [0,1] \backslash S : \delta(t) = 0 \} \cup S$. Let 
$E_{\gamma} = \limsup_{j \rightarrow +\infty}\limits \bigcup_{i=2^j}^{2^{j+1}-1}\limits \left[ V_i - i^{-\frac{1}{\gamma \alpha(V_i)}} , V_i + i^{-\frac{1}{\gamma \alpha(V_i)}}\right].$ 
Since $E_{\gamma} \subset \{ t : \h \leq \gamma \}$, $$\{ t \in [0,1] \backslash S : \delta(t) = 0 \} \subset \bigcap_{\gamma >0}\limits E_{\gamma}.$$

Now, $\left(\bigcup_{i\geq 1}\limits \left[ V_i - i^{-\frac{1}{\gamma d}} , V_i + i^{-\frac{1}{\gamma d}}\right]\right)_i$ 
is a covering of $E_{\gamma}$, and thus $\dim_H (E_{\gamma}) \leq \gamma d$.
As a consequence, $\dim_H (\{ t \in [0,1] \backslash S : \delta(t) = 0 \} ) =0$. Since $\dim_H (S) =0$, we find that
$f_H(0)=0$.

\bigskip

\noindent  {\bf Proof of Lemma \ref{lem8}:}

Following \cite{Bar}, set $\lambda_i = \frac{1}{1+\Gamma_i}$. For the system of points 
$\mathcal{P}= \{ (V_i, \lambda_i )\}_{i \geq 1}$ and $t \in [0,1]$, define the approximation rate of $t$ 
by $\mathcal{P}$ as

$$\delta_t(\mathcal{P}) = \sup \{ \delta \geq 1 : t \textrm{ belongs to an infinite number of balls } B(V_i, \lambda_i^{\delta}) \}.$$

Let us show that $\delta_t(\mathcal{P}) = \frac{1}{\delta(t)}$. 

In that view, note first that, since $\lim_{i \rightarrow +\infty}\limits \frac{\Gamma_i}{i} =1$ almost surely,

$$\delta(t) = \inf_{(V_{\phi(n)})_n \in \mathcal{R}_t}\limits \liminf_{i \infty}\limits -\frac{\log | 1+ \Gamma_{\phi(i)}|}{\log |V_{\phi(i)} -t|}.$$

Let $\delta \geq 1$ be such that $t$ belongs to an infinite number of balls $B(V_i, \lambda_i^{\delta})$. 
There exists $\phi : \bbbn \rightarrow \bbbn$ such that 
$(V_{\phi(n)})_n \in \mathcal{R}_t$ and $|t-V_{\phi(i)}| \leq \lambda_{\phi(i)}^{\delta}$, {\it i.e.} $-\frac{\log | 1+ \Gamma_{\phi(i)}|}{\log |V_{\phi(i)} -t|} \leq \frac{1}{\delta}$. As a consequence $\delta(t) \leq \frac{1}{\delta}$. 
Taking the supremum over all admissible $\delta$, one gets $\delta_t(\mathcal{P}) \leq \frac{1}{\delta(t)}$.

For the reverse inequality, consider two cases:
\begin{itemize}
 \item $\delta(t) = 1$: since $\delta_t(\mathcal{P}) \geq 1$, one gets $\delta_t(\mathcal{P}) \geq \frac{1}{\delta(t)}$;
 \item $\delta(t) < 1$: choose $\varepsilon > 0$ such that $\delta(t) + \varepsilon \leq 1$.
There exists $\phi : \bbbn \rightarrow \bbbn$ such that for all large enough $i \in \bbbn$, $-\frac{\log | 1+ \Gamma_{\phi(i)}|}{\log |V_{\phi(i)} -t|} \leq \delta(t) + \varepsilon$, {\it i.e.}
$| t- V_{\phi(i)}| \leq \frac{1}{( 1+ \Gamma_{\phi(i)})^{\frac{1}{\delta(t) + \varepsilon}}}$.
By definition, this entails $\delta_t(\mathcal{P}) \geq \frac{1}{\delta(t)+\varepsilon}$, and finally
$\delta_t(\mathcal{P}) \geq \frac{1}{\delta(t)}$ by letting $\varepsilon$ go to 0.

\end{itemize}

We now apply Theorem 21 of \cite{Bar}: since $hd \leq 1$, one has $h\alpha(t) \leq 1$ for all $t \in (0,1)$. 
The function $t \mapsto \frac{1}{h \alpha(t)}$ is continuous and thus 
$$\dim_H \left(\{ t \in (0,1) : \delta_t(\mathcal{P}\right) = \frac{1}{h \alpha(t)}\}) = \sup \{ h \alpha(t), t \in (0,1) \} = hd$$
i.e. $f_H(h)=hd$ \Box

\bigskip

\noindent  {\bf Proof of Lemma \ref{lem9}:}

Since $hc<1<hd$ and $\alpha$ is $C^1$, there exist $(t_0,t_1) \in (0,1)^2$ such that $h \alpha(t_0) =1$ and
$h\alpha(t) \leq 1$ for all $t \in I:=(t_1,t_0)$ or $(t_0,t_1)$.

Define $g : t \mapsto \min (1, \frac{1}{h \alpha(t)} ).$ Theorem 21 in \cite{Bar} yields: 
\begin{eqnarray*}
 \dim_H (\{ t \in I : \delta_t(\mathcal{P}) = g(t) \}) & = & \sup \{ \frac{1}{g(t)}, t \in I \} \\
 & = & \sup \{ h \alpha(t), t \in I \} \\
 & = & h \alpha(t_0)\\
 & = & 1.
\end{eqnarray*}
One gets that $f_H(h) \geq \dim_H (\{ t \in I : \delta(t) = h \alpha(t) \}) =1$, {\it i.e.} $f_H(h)=1$. \Box

\bigskip

\noindent  {\bf Proof of Lemma \ref{lem10}:}

By definition and Lemma \ref{lem2}, 
$$F_{1/c} = \{ t \in [0,1] : \delta(t) = \frac{\alpha(t)}{c} \} = \{ t \in [0,1] : \alpha(t) = c \} \cap \{ t \in [0,1] : \delta(t) = 1 \}.$$

\noindent Set $E = \{ t \in [0,1] : \alpha(t) = c \}$, $E_0 =\{ t \in [0,1] : \delta(t) < 1 \}$ and $ E_1 =  \{ t \in [0,1] : \delta(t) = 1 \}$.
Then $[0,1] = E_0 \cup E_1$, $F_{1/c} = E \cap E_1$ and thus $f_H(\frac{1}{c}) \leq \dim_H(E)$.

\noindent Now, if $\dim_H(E) = 0$, the lemma holds true since $F_{1/c}$ is not empty and thus $f_H(1/c) \geq 0$.

\noindent Suppose then that $\dim_H(E) >0$. Choose $s < \dim_H(E)$. This implies that
$\mathcal{H}^s(E) = +\infty$, and Theorem 4.10 in \cite{Falc} entails that there exist a compact set $E_c \subset E$ 
such that $0 < \mathcal{H}^s(E_c) < +\infty$.

Set $\mu_s(.) = \mathcal{H}^s(E_c \cap .)$. This is a finite and positive Borel measure on $[0,1]$. 
Theorem 3.7 in \cite{LL2}, along with Lemmas \ref{lem3} and \ref{lem5}, entail that 
for all $t \in (0,1)$, $\P (t \in E_1) =1$ and thus $\P ( t \in E_c \cap E_1) = \mathbf{1}_{t \in E_c}$.

One computes:
\begin{eqnarray*}
 \E \left[ \mu_s( E_0 ) \right] & = & \E \left[ \int_0^1 \mathbf{1}_{t \in E_0} \mu_s(dt) \right]\\
 & = & \int_0^1 \E \left[ \mathbf{1}_{t \in E_0}\right] \mu_s(dt)\\
 & = & \int_{E_c} \P(t \in E_0 ) \mu_s(dt)\\
 & = & 0.
\end{eqnarray*}
Thus, $\mu_s( E_0 )$ is a positive random variable with vanishing expectation: almost surely, $\mu_s( E_0 )=0$. 
Since $\mu_s(E_c) = \mu_s(E_c \cap E_0 ) + \mu_s( E_c \cap E_1),$ one obtains that, almost surely, 
$\mu_s(E_c) =  \mu_s( E_c \cap E_1)$.

Now, $\mathcal{H}^s( E \cap E_1) \geq \mathcal{H}^s( E_c \cap E_1) = \mathcal{H}^s( E_c ) >0$. 
Thus $\dim_H(E \cap E_1) \geq s$, $\forall s \leq \dim_H(E)$ and $\dim_H(E \cap E_1) \geq  \dim_H(E)$. \Box

\bigskip

\noindent  {\bf Proof of Lemma \ref{lem11}:}

For all $t \in [0,1]$, $\alpha(t) \geq c$ and almost surely, for all $t \in [0,1]$, $\delta(t) \leq 1$, thus, almost surely,
for all $t \in [0,1]$, $\h \leq \frac{1}{c}$. As a consequence, for $h > \frac{1}{c}$, $F_h = \varnothing$ and $f_H(h) = -\infty$. \Box

\section{Large deviation and Legendre multifractal spectra}
We compute in this section the large deviation and Legendre multifractal spectra of the process $B$ on an interval. 
Recall that we consider the process on $[0,1]$, and that
the large deviation multifractal spectrum of a process $X$ on $[0,1]$ is the (random) function $f_g$ defined on $\bbbr$ by
$$f_g(\beta) = \lim_{\varepsilon \rightarrow 0^+}\limits \liminf_{n \rightarrow +\infty}\limits \frac{\log \Nb}{\log n},$$
where, for a positive integer $n$ and $\varepsilon>0$, 
$$\Nb = \# \{ j \in \{0, \ldots, n-1 \}: \beta - \varepsilon \leq \frac{\log |X(\frac{j+1}{n}) - X(\frac{j}{n}) |}{- \log n} \leq \beta + \varepsilon \}.$$
Other large deviation multifractal spectra can be defined by replacing the increments $X(\frac{j+1}{n}) - X(\frac{j}{n})$
by other measures of the variation of $X$, such as its oscillations, but we will not consider these in this work.
 We do not recall the definition of the Legendre multifractal spectrum, and refer the reader to \cite{Falc,LVV}
instead.

We shall denote 
\begin{eqnarray*}
P_n^j & = & \P \left(\beta - \varepsilon \leq \frac{\log |Y(\frac{j+1}{n}) - Y(\frac{j}{n}) |}{- \log n} \leq \beta + \varepsilon \right) \\
& = & \P \left( \frac{1}{n^{\beta + \varepsilon}} \leq | Y(\frac{j+1}{n}) - Y(\frac{j}{n}) | \leq \frac{1}{n^{\beta - \varepsilon}} \right).
\end{eqnarray*}
 
Set also

$$X_j = \mathbf{1}_{\{ \frac{1}{n^{\beta + \varepsilon}} \leq | B(\frac{j+1}{n}) - B(\frac{j}{n}) | \leq \frac{1}{n^{\beta - \varepsilon}} \} }$$ 
which follows a Bernoulli law with parameter $P_n^j$. Clearly,  
$$\Nb = \sum_{j=1}^{n}\limits X_j.$$
For $U$ an open interval of $(0,1)$, we write
$$J_n(U) = \{ j : \frac{j}{n} \in U \}.$$
There exists a constant $K_U >0$ such that, for $n$ large enough, 
$$\#J_n(U) \geq K_U n.$$

Finally, we will make use of the characteristic function of $B$, which reads \cite{LGLVL}:
%\begin{equation}\label{FBCF}
%\E \left( \exp \left(i \sum_{j=1}^{m}\limits \theta_j L_{F}(t_j) \right) \right)= \exp \left( -2 \int_{[0,T]} \int_{0}^{+ \infty} \sin^2\left( \sum_{j=1}^{m} \theta_j \frac{C_{\alpha(t_j)}^{1/\alpha(t_j)}}{2y^{1/\alpha(t_j)}} 1_{[0,t_j]}(x) \right)\hspace{0.1cm} dy \hspace{0.1cm} dx \right)
%\end{equation}
%where $ m \in \mathbb{N}, (\theta_1, \ldots, \theta_m) \in \bbbr^m, (t_1, \ldots , t_m) \in \bbbr^m$.
%
%\bigskip
%In \cite{KL}, another path is followed to define a multistable extension of L\'evy motion.
%Considering the characteristic function of $\alpha-$stable L\'evy motion $L$:
%\begin{equation*}
%\mathbb{E}\left(\exp(i\theta L(t))\right)= \exp \left(-t |\theta|^{\alpha} \right),
%\end{equation*}
%one defines the process $L_{I}$ by its joint characteristic function as follows:
\begin{equation}\label{IICF}
\mathbb{E}\left(\exp\left(i\sum_{j=1}^m\theta_j B(t_j)\right)\right)=\exp\left(-\int\left|\sum_{j=1}^m \theta_j1_{[0,t_j]}(s)\right|^{\alpha(s)} ds\right).
\end{equation}
where $ m \in \mathbb{N}, (\theta_1, \ldots, \theta_m) \in \bbbr^m, (t_1, \ldots , t_m) \in \bbbr^m$.

\subsection{Main result}
The large deviation and Legendre multifractal spectra of $B$ are described by the following theorem:
\begin{theo}\label{spectreLD}
With probability one, the large deviation and Legendre multifractal spectra of $B$ satisfy:
\begin{equation}\label{FGFL}
f_g(\beta)=f_l(\beta)=
\begin{cases}
-\infty &\text{for $\beta<0$};\\
\beta d &\text{for $\beta \in [0,\frac{1}{d}]$};\\
1 &\text{for $\beta \in (\frac{1}{d},\frac{1}{c}]$};\\
1 + \frac{1}{c} - \beta &\text{for $\beta \in (\frac{1}{c},1+\frac{1}{c}]$};\\
-\infty & \text{for $\beta>\frac{1}{c}$}.
\end{cases}
\end{equation} 
\end{theo}

\noindent The fact that $f_l=f_g$ stems from the general result that $f_l$ is always the concave hull of $f_g$ when the
set $\{\beta: f_g(\beta) \geq 0 \}$ is bounded.
The part concerning $f_g$ in Theorem \ref{spectreLD} follows from a series of lemmas that are proven in the next sections.

We note in passing that, comparing with Theorem \ref{spectre}, we see that the weak multifractal formalism
holds for $B$, but the strong one does not, that is, $f_H \leq f_g$ and $f_H \neq f_g$. The decreasing part
with slope -1 for ``large'' exponents present in $f_g$ but not in $f_H$ 
is a common phenomenon when variations are measured with increments.

In order to prove Theorem \ref{spectreLD}, we will first show in each case of \eqref{FGFL} that the equality holds true for any 
given $\beta$ with probability one. Permuting ``for all $\beta$'' and ``almost surely'' will then often 
be achieved thanks to the two following general simple but useful lemmas on the large deviation spectrum, 
which are of independent interest.

\begin{lem}\label{inversionlimsup}
 The large deviation spectrum of any real function is an upper semicontinuous function. 

\end{lem}

 \begin{proof}
 Let $f_g$ be the large deviation spectrum of a real function. 
Consider $\beta \in \bbbr$, $(x_j)_{j \geq 1}$ a sequence such that $\lim_{j \rightarrow +\infty}\limits x_j = \beta$, and 
set $\varepsilon_j = \sup_{k \geq j}\limits |\beta-x_k|.$

 Note first that
 
 $$\lim_{\varepsilon \rightarrow 0}\limits \liminf_{n \rightarrow +\infty}\limits \frac{\log N^{\varepsilon}_n(\beta)}{\log n} = \lim_{j \rightarrow +\infty}\limits \liminf_{n \rightarrow +\infty}\limits \frac{\log N^{\varepsilon_j}_n(\beta)}{\log n} .$$

For all $j\geq 1$ and all $l \geq j$,
$N^{2\varepsilon_j}_n(\beta) \geq N^{\varepsilon_j}_n(x_j) \geq N^{\varepsilon_l}_n(x_j)$. As a consequence,

$$\liminf_{n \rightarrow +\infty}\limits \frac{\log N^{2\varepsilon_j}_n(\beta)}{\log n} \geq \liminf_{n \rightarrow +\infty}\limits \frac{\log N^{\varepsilon_l}_n(x_j)}{\log n}.$$
Letting $l$ tend to infinity, one gets $\liminf_{n \rightarrow +\infty}\limits \frac{\log N^{2\varepsilon_j}_n(\beta)}{\log n} \geq f_g(x_j)$ and letting $j$ tend to infinity one finally obtains 
$$f_g(\beta) \geq \limsup_{j \rightarrow +\infty}\limits f_g(x_j).$$

\end{proof}

\begin{lem}\label{inversionencadrement}

Assume that there exist four functions $\underline{h}, \overline{h},\underline{g}, \overline{g}$ 
with $\lim_{u \to 0} \limits \ \underline{g}(u)=\lim_{u \to 0} \limits \ \overline{g}(u)=0$
such that, for all $\beta$ in some interval $I$ and all sufficiently small $\varepsilon>0$, almost surely 
$$\underline{h}(\beta) + \underline{g}(\varepsilon) \leq \liminf_{n \rightarrow +\infty}\limits \frac{\log \Nb}{\log n} 
\leq \overline{h}(\beta) + \overline{g}(\varepsilon).$$
Then, almost surely, for all $\beta$ in $I$, $\underline{h}(\beta) \leq f_g(\beta) \leq \overline{h}(\beta)$.

\end{lem}

\begin{proof}

Define
$$M_n(\beta_1,\beta_2) = \# \{ j \in \{0, \ldots, n-1 \}: \beta_1 \leq \frac{\log |X(\frac{j+1}{n}) - X(\frac{j}{n}) |}{- \log n} \leq \beta_2\}.$$
Then, for $\beta_1< \beta < \beta_2$,
$$N_n^{(\beta_2-\beta) \vee (\beta-\beta_1)}(\beta) \leq M_n(\beta_1,\beta_2) \leq N_n^{(\beta_2-\beta) \wedge (\beta-\beta_1)}(\beta)$$
and
$$f_g(\beta) = \inf_{(\beta_1, \beta_2) \in \mathbb{Q}^2, \beta_1<\beta<\beta_2} \liminf_{n \rightarrow +\infty}\limits 
\frac{\log M_n(\beta_1,\beta_2)}{\log n}.$$

The set
$$A = \bigcup_{(\beta_1,\beta_2) \in \mathbb{Q}^2} M(\beta_1,\beta_2)$$
is countable, and, thus $f_g$ is obtained by a countable infimum for all $\beta \in I$. This yields the result.

\end{proof} 

\subsection{Preliminary lemmas}

\subsubsection{Statements}

Define $H(t) = H_{\lambda,p}(t)=\lambda t - \log (1 - p + p e^t ).$

\begin{lem}\label{LemPR1}
 If $0<p<\lambda < 1$, then
 
$$\sup_{t > 0}\limits H(t) = \lambda \log (\frac{\lambda}{p}) + (1-\lambda) \log(\frac{1-\lambda}{1-p} ).$$ 
\end{lem}

\begin{lem}\label{LemPR2}
 If $0<\lambda < p < 1$, then
 
$$\sup_{t < 0}\limits H(t) = \lambda \log (\frac{\lambda}{p}) + (1-\lambda) \log(\frac{1-\lambda}{1-p} ).$$ 
\end{lem}

\begin{lem}\label{LemPR3}
 If $p=p(n)=K n^b$ and $\lambda = \lambda(n)= n^a$, where $K>0$ and $0>a>b$, then
 there exists $n_0 \in \bbbn$ such that, for all $n \geq n_0,$
 
 $$\sup_{t > 0}\limits H(t) \geq \frac{(a-b)}{2} n^a \log n.$$
\end{lem}

\begin{lem}\label{LemPR4}
 If $p=p(n)=K_1 n^b$ and $\lambda = \lambda(n)=K_2 n^a$, where
  $K_1 >0$, $K_2 >0$ and $0>b>a$, then there exists $n_0 \in \bbbn$ such that, for all $n \geq n_0,$
 
 $$\sup_{t < 0}\limits H(t) \geq \frac{K_1}{4} n^b.$$
\end{lem}

\begin{lem}\label{LemPR5}
 If $p=p(n)=K n^b$ and $\lambda = \lambda(n)= n^a$, where $K>0$ and $0>a>b$, then
 there exists $n_0 \in \bbbn$ such that, for all $n \geq n_0,$
  
 $$\inf_{t > 0}\limits e^{- \lambda t n} (1-p+p e^t )^n \leq e^{-\frac{(a-b)}{2} n^{1+a} \log n}.$$
\end{lem}

\begin{lem}\label{LemPR6}
 If $p=p(n)=K_1 n^b$ and $\lambda = \lambda(n)=K_2 n^a$, where
  $K_1 >0$, $K_2 >0$ and $0>b>a$, then there exists $n_0 \in \bbbn$ such that, for all $n \geq n_0,$

  $$\inf_{t < 0}\limits e^{- \lambda t n} (1-p+p e^t )^{K_U n} \leq e^{-\frac{K_1 K_U}{4} n^{1+b}}.$$
\end{lem}

\begin{lem}\label{LemPR7}
 Assume there exist $b \in (-1,0)$ and $K>0$ such that,  for all $n \geq n_0$ and 
 for all $j \in \llbracket 1,n \rrbracket $, $P_n^j \leq K n^b$. Then, almost surely,
 
$$ \quad \liminf_{n \rightarrow +\infty}\limits \frac{\log \Nb}{\log n} \leq 1+b.$$ 

\end{lem}

\begin{lem}\label{LemPR8}
 Assume there exist an open interval $U$, a real $b \in (-1,0)$ and $K>0$ such that, for all $n \geq n_0$ and
 for all $j \in J_n(U)$, $P_n^j \geq K n^b$. Then, almost surely,
$$\quad \liminf_{n \rightarrow +\infty}\limits \frac{\log \Nb}{\log n} \geq 1+b.$$ 

\end{lem}

\subsubsection{Proofs}

\begin{proof}[Proof of Lemma \ref{LemPR1}]

$\sup_{t >0}\limits H(t) = H(t_0)$ where $t_0 = \log \left( \frac{\lambda (1-p)}{p(1-\lambda)}\right).$
 
\end{proof}

\begin{proof}[Proof of Lemma \ref{LemPR2}]

$\sup_{t <0}\limits H(t)= H(t_0)$ where $t_0 = \log \left( \frac{\lambda (1-p)}{p(1-\lambda)}\right).$
 
\end{proof}

\begin{proof}[Proof of Lemma \ref{LemPR3}]

With $t_0 = \log \left( \frac{\lambda (1-p)}{p(1-\lambda)}\right)$, one has, for $n$ large enough,
\begin{eqnarray*}
 H(t_0) & = & n^a \log \left(\frac{n^{a-b}}{K}\right) + (1-n^a) \log \left(\frac{1-n^a}{1-Kn^b}\right)\\
 & = & n^a \log \left(\frac{n^{a-b}}{K}\right) + (1-n^a) \log \left(1+ \frac{Kn^b - n^a}{1-Kn^b}\right)\\
 & \geq & n^a \log \left(\frac{n^{a-b}}{K}\right) + 2(1-n^a) \left(\frac{Kn^b - n^a}{1-Kn^b}\right)\\
 & = & (a-b) n^a \log n + n^a \left[ 2(1-n^a) \frac{Kn^{b-a} - 1}{1-Kn^b} - \log K \right].
\end{eqnarray*}

Since $\lim_{n \rightarrow +\infty}\limits 2(1-n^a) \frac{Kn^{b-a} - 1}{1-Kn^b} - \log K = -2 - \log K$,
 there exists $n_0 \in \bbbn$ such that, for all $n \geq n_0$,
$$ \frac{n^{-a}H(t_0)}{\log n} \geq (a-b) + \frac{1}{\log n}\left[ 2(1-n^a) \frac{Kn^{b-a} - 1}{1-Kn^b} - \log K \right] \geq \frac{(a-b)}{2}.$$

\end{proof}

\begin{proof}[Proof of Lemma \ref{LemPR4}]

With $t_0 = \log \left( \frac{\lambda (1-p)}{p(1-\lambda)}\right)$, one has, for $n$ large enough,
\begin{eqnarray*}
 H(t_0) & = & K_2 n^a \log \left(\frac{K_2 n^{a-b}}{K_1}\right) + (1-K_2 n^a) \log \left(1+ \frac{K_1 n^b - K_2 n^a}{1-K_1 n^b}\right)\\
 & \geq & K_2 n^a \log \left(\frac{K_2 n^{a-b}}{K_1}\right) + \frac{1}{2}(1-K_2 n^a) \left(\frac{K_1 n^b - K_2 n^a}{1-K_1 n^b}\right)\\
 & = & n^b \left[ K_2 n^{a-b} \log \left(\frac{K_2 n^{a-b}}{K_1}\right) + \frac{1}{2}(1-K_2 n^a) \frac{K_1 - K_2 n^{a-b}}{1-K_1 n^b} \right]\\
 \end{eqnarray*}
which yields the result since $\lim_{n \rightarrow +\infty}\limits K_2 n^{a-b} \log (\frac{K_2 n^{a-b}}{K_1}) + \frac{1}{2}(1-K_2 n^a) \frac{K_1 - K_2 n^{a-b}}{1-K_1 n^b} = \frac{K_1}{2}.$
\end{proof}

\begin{proof}[Proof of Lemma \ref{LemPR5}]
 One has 
 \begin{eqnarray*}
  \inf_{t > 0}\limits e^{- \lambda t n} (1-p+p e^t )^n & = & \inf_{t > 0}\limits e^{-n H(t)}\\
  & = & e^{-n \sup_{t > 0}\limits H(t)}.
 \end{eqnarray*}
Lemma \ref{LemPR3} then implies that, for $n \geq n_0$,
$$\inf_{t > 0}\limits e^{- \lambda t n} (1-p+p e^t )^n \leq e^{- \frac{(a-b)}{2} n^{1+a} \log n}.$$
\end{proof}

\begin{proof}[Proof of Lemma \ref{LemPR6}]
 Write 
\begin{eqnarray*}
  \inf_{t < 0}\limits e^{- \lambda t n} (1-p+p e^t )^{K_U n} & = & e^{- K_U n \sup_{t <0}\limits H(t)}
  \end{eqnarray*}
  where $H(t) = \frac{\lambda}{K_U} t - \log (1 - p + p e^t ).$ Lemma \ref{LemPR4} ensures that, for $n \geq n_0$,
$$\inf_{t < 0}\limits e^{- \lambda t n} (1-p+p e^t )^n \leq e^{-\frac{K_1 K_U}{4} n^{1+b}}.$$

\end{proof}

 \begin{proof}[Proof of Lemma \ref{LemPR7}]
  Fix $a \in (b,0)$. Then, for all $t>0$,
  
  \begin{eqnarray*}
\P \left( \Nb \geq n^{1+a} \right) & = &  \P \left(e^{t \sum_{j=1}^n\limits X_j} \geq e^{t n^{1+a}} \right)\\
& \leq & e^{- n t \lambda} \E \left[ \prod_{j=1}^{n} e^{t X_j}\right]
  \end{eqnarray*}
where $\lambda = n^a$. The $X_j$ are independent and $\E \left[ e^{t X_j}\right] = 1- P_n^j + P_n^j e^t$, 
thus
$$\P \left( \Nb \geq n^{1+a} \right) \leq e^{- n t \lambda}\prod_{j=1}^{n} (1- P_n^j + P_n^j e^t), \quad \forall t>0.$$ 
For $t>0$, the function $p \mapsto 1-p+pe^t$ is increasing and so, by assumption on $P_n^j$,
$$\P \left( \Nb \geq n^{1+a} \right) \leq e^{- n t \lambda} (1- p + p e^t)^n,$$ 
where $p=K n^b$.
Minimizing over $t>0$ and using Lemma \ref{LemPR5}, one gets
$$\forall n \geq n_0, \P \left( \Nb \geq n^{1+a} \right) \leq e^{- \frac{(a-b)}{2} n^{1+a} \log n}$$
and thus
$$\sum_{n \in \bbbn}\limits \P \left( \Nb \geq n^{1+a} \right) < +\infty.$$
The Borel-Cantelli lemma then ensures that, almost surely,
$$\exists n_0 \in \bbbn, \forall n \geq n_0, \Nb \leq n^{1+a}$$ 
or
$$\liminf_{n \rightarrow +\infty}\limits \frac{\log \Nb}{\log n} \leq 1+a.$$
Since this inequality holds true for any $a \in (b,0)$ one has indeed that, almost surely, 
$$\liminf_{n \rightarrow +\infty}\limits \frac{\log \Nb}{\log n} \leq 1+b.$$

 \end{proof}

  \begin{proof}[Proof of Lemma \ref{LemPR8}]
  Fix $a <b$. For all $t <0$, 
    \begin{eqnarray*}
\P \left( \Nb \leq n^{1+a} \right) & \leq & \P \left( \sum_{j \in J_n(U)}\limits X_j \leq n^{1+a}\right)\\ 
& = &  \P \left(e^{t \sum_{j \in J_n(U)}\limits X_j} \geq e^{t n^{1+a}} \right)\\
& \leq & e^{-  t n^{1+a}}\prod_{j \in J_n(U)}\limits (1- P_n^j + P_n^j e^t).
  \end{eqnarray*}
  
  When $t<0$, the function $p \mapsto 1-p+pe^t$ is decreasing. As a consequence, by assumption on $P_n^j$
  and with $p=Kn^b$, $\lambda = n^a$, 
  one has, for $n$ large enough,
    \begin{eqnarray*}
     \P \left( \Nb \leq n^{1+a} \right) & \leq &  e^{-  n t \lambda} (1-p + pe^t)^{\# J_n(U)} \\
     & \leq & e^{-  n t \lambda} (1-p + pe^t)^{K_U n}.
    \end{eqnarray*}
Minimizing over $t<0$ and using Lemma \ref{LemPR6}, one gets
$$\forall n \geq n_0, \P \left( \Nb \leq n^{1+a} \right) \leq e^{-\frac{K_1 K_U}{4} n^{1+b}}$$
and thus 
$$\sum_{n \in \bbbn}\limits \P \left( \Nb \leq n^{1+a} \right) < +\infty.$$
As in the proof of Lemma \ref{LemPR7}, this leads to 
$$\forall a<b, \quad \mbox{almost surely,} \quad \liminf_{n \rightarrow +\infty}\limits \frac{\log \Nb}{\log n} \geq 1+a$$ 
and finally, almost surely,
$$\liminf_{n \rightarrow +\infty}\limits \frac{\log \Nb}{\log n} \geq 1+b.$$
   \end{proof}

  \subsection{Estimates of $P_n^j$}
  
  For $U$ an open interval, denote $c_U = \inf_{t \in U} \alpha(t)$ and $d_U = \sup_{t \in U} \alpha(t)$. 
  Set also $t_j = \frac{j}{n}.$
  
  \subsubsection{Lemmas}
  
  \begin{lem}\label{LemE1maj}
   Assume $\beta < \frac{1}{d}$. Then, $\exists K>0$ , $\exists n_0 \in \bbbn$ such that $\forall n \geq n_0$, $\forall j \in \llbracket 1,n \rrbracket$, 
   $$P_n^j \leq K n^{\alpha(t_j) \beta + \alpha(t_j) \varepsilon -1}.$$
   
  \end{lem}
 \begin{lem}\label{LemE1min}
   Assume $\beta < \frac{1}{d_U}$. Then, $\exists K>0$ , $\exists n_0 \in \bbbn$ such that $\forall n \geq n_0$, $\forall j \in J_n(U)$, 
   $$P_n^j \geq K n^{\alpha(t_j) \beta - \alpha(t_j) \varepsilon -1}.$$
   
  \end{lem}
  
    \begin{lem}\label{LemE2maj}
   Assume $\beta > \frac{1}{c}$ and $\varepsilon \in (0, \beta - \frac{1}{c})$. Then, $\exists K>0$ , $\exists n_0 \in \bbbn$ such that $\forall n \geq n_0$, $\forall j \in \llbracket 1,n \rrbracket$, 
   $$P_n^j \leq K n^{\frac{1}{\alpha(t_j)} + \varepsilon -\beta }.$$
   
  \end{lem}
 \begin{lem}\label{LemE2min}
   Assume $\beta > \frac{1}{c_U}$ and $\varepsilon \in (0, \beta - \frac{1}{c_U})$. Then, $\exists K>0$ , $\exists n_0 \in \bbbn$ such that $\forall n \geq n_0$, $\forall j \in J_n(U)$, 
   $$P_n^j \geq K n^{\frac{1}{\alpha(t_j)} + \varepsilon -\beta }.$$
   
  \end{lem}
  
   \subsubsection{Proofs}
  
    \begin{proof}[Proof of Lemma \ref{LemE1maj}]
     Set $\mu_j = \alpha(t_j) \beta + \alpha(t_j) \varepsilon -1$. Using the truncation
     inequality \cite[Section 13, p.\ 209]{Loeve}, one computes
     
     \begin{eqnarray*}
      P_n^j & = & \P \left( \frac{1}{n^{\beta + \varepsilon}} \leq | Y(t_{j+1}) - Y(t_j) | \leq \frac{1}{n^{\beta - \varepsilon}} \right)\\
      & \leq & \P \left(  | Y(t_{j+1}) - Y(t_j) | \geq  \frac{1}{n^{\beta + \varepsilon}} \right)\\
      & \leq & \frac{7}{n^{\beta+\varepsilon}} \int_{0}^{n^{\beta+\varepsilon}}\limits \left( 1- e^{-\int_{t_j}^{t_{j+1}} |\xi|^{\alpha(x)} dx} \right) d\xi\\
      & = & 7 \int_{0}^{1}\limits \left( 1- e^{-\int_{t_j}^{t_{j+1}} n^{(\beta+\varepsilon)\alpha(x)} |v|^{\alpha(x)} dx }\right) dv\\
      & \leq & 7  \int_{0}^{1}\limits  \int_{t_j}^{t_{j+1}}\limits n^{(\beta+\varepsilon)\alpha(x)} |v|^{\alpha(x)} dx dv\\
      & \leq & 7  \int_{t_j}^{t_{j+1}}\limits n^{(\beta+\varepsilon)\alpha(x)}  dx\\
      & = & 7 n^{\mu_j} \int_{t_j}^{t_{j+1}}\limits n^{1+(\beta+\varepsilon)(\alpha(x)-\alpha(t_j) )}  dx.
     \end{eqnarray*}
Since $\alpha$ is $C^1$, there exists a constant $K$ such that, for all $x \in (t_j,t_{j+1})$,
$$|\beta+\varepsilon||\alpha(x)-\alpha(t_j)| \leq K |x-t_j| \leq \frac{K}{n}.$$ 
As a consequence,
\begin{eqnarray*}
  P_n^j & \leq & 7 n^{\mu_j} \int_{t_j}^{t_{j+1}}\limits n^{1+\frac{K}{n}}  dx \\
  & = & 7 n^{\mu_j} n^{\frac{K}{n}} \\
  & \leq & K_1 n^{\mu_j}
 \end{eqnarray*}
for a constant $K_1$.

\end{proof}

\begin{proof}[Proof of Lemma \ref{LemE1min}]
 Set $\mu = \frac{1}{2} \left( \frac{1}{n^{\beta + \varepsilon}} + \frac{1}{n^{\beta - \varepsilon}} \right)$ 
 and $\sigma = \frac{1}{2} \left( \frac{1}{n^{\beta - \varepsilon}} - \frac{1}{n^{\beta + \varepsilon}}\right)$. 

 Choose a function $\varphi$ that satisfies the following properties:
 
 \begin{enumerate}
  \item $\textrm{supp}(\varphi) \subset [-1,1]$.
  
  \item $\varphi$ is even.
  
  \item $\varphi$ is $\mathcal{C}^4$.
  
  \item $\forall x \in \bbbr$, $0 \leq \varphi(x) \leq 1$.
  
  \item $\varphi$ is not identically $0$. 
 \end{enumerate}

 These properties imply in particular that $\hat{\varphi}$ is real and even. In addition, for all $x \in \bbbr$:
 
 $$\varphi \left( - \frac{x+\mu}{\sigma} \right) + \varphi \left( \frac{x-\mu}{\sigma}\right) \leq \mathbf{1}_{[-n^{-\beta+\varepsilon} , -n^{-\beta - \varepsilon}]}(x) + \mathbf{1}_{[n^{-\beta-\varepsilon}, n^{-\beta+\varepsilon} ]}(x).$$
 
 Since the Fourier transform of $\varphi \left( \frac{x-\mu}{\sigma}\right)$ is 
 $\sigma \exp(- i \mu \xi) \hat{\varphi}(\sigma \xi)$, Parseval formula yields
 
 \begin{eqnarray*}
  P_n^j & \geq & \int_{\bbbr}\limits  \left(\varphi ( - \frac{x+\mu}{\sigma} ) + \varphi ( \frac{x-\mu}{\sigma}) \right) \P( Y(t_{j+1}) - Y(t_j) \in dx )\\
  & \geq & \frac{1}{\pi} \sigma \int_{\bbbr}\limits \cos(\mu \xi ) \hat{\varphi}(\sigma \xi) e^{-\int_{t_j}^{t_{j+1}} |\xi|^{\alpha(x)} dx} d\xi\\
  & = & \frac{2}{\pi} \int_{0}^{+\infty}\limits \cos(\frac{\mu}{\sigma} \eta ) \hat{\varphi}(\eta) e^{-\int_{t_j}^{t_{j+1}} |\frac{\eta}{\sigma}|^{\alpha(x)} dx} d\eta.
 \end{eqnarray*}

% REMPLACER de LA : 
% 
%  The integral of $\cos(\eta) \hat{\varphi}(\eta)$ is zero, since it is the value of the Fourier transform 
%  of $\hat{\varphi}$ at point $1$, that is, $\varphi(1)$, which is zero by assumption on $\varphi$. One may
%  thus write
% 
% \begin{eqnarray*}
%   P_n^j & \geq &\frac{2}{\pi} \int_{0}^{+\infty}\limits \left( \cos(\frac{\mu}{\sigma} \eta ) \hat{\varphi}(\eta) e^{-\int_{t_j}^{t_{j+1}} |\frac{\eta}{\sigma}|^{\alpha(x)} dx} - \cos(\eta) \hat{\varphi}(\eta) \right) d\eta\\
%   & = & \frac{2}{\pi} \int_{0}^{+\infty}\limits \cos(\frac{\mu}{\sigma} \eta ) \hat{\varphi}(\eta) (e^{-\int_{t_j}^{t_{j+1}} |\frac{\eta}{\sigma}|^{\alpha(x)} dx} - 1) d\eta + \frac{2}{\pi} \int_{0}^{+\infty}\limits \hat{\varphi}(\eta) (\cos(\frac{\mu}{\sigma} \eta ) - \cos(\eta) ) d\eta\\
%  & = & \frac{2}{\pi} \int_{0}^{+\infty}\limits \cos(\frac{\mu}{\sigma} \eta ) \hat{\varphi}(\eta) (e^{-\int_{t_j}^{t_{j+1}} |\frac{\eta}{\sigma}|^{\alpha(x)} dx} - 1) d\eta + \frac{1}{\pi} (\varphi(\frac{\mu}{\sigma}) - \varphi(1) )\\
%  & \geq & \frac{2}{\pi} \int_{0}^{+\infty}\limits \cos(\frac{\mu}{\sigma} \eta ) \hat{\varphi}(\eta) (e^{-\int_{t_j}^{t_{j+1}} |\frac{\eta}{\sigma}|^{\alpha(x)} dx} - 1) d\eta\\
%  & = & \frac{2}{\pi} n^{\alpha(t_j) \beta - \alpha(t_j) \varepsilon - 1} I_n^j,
%  \end{eqnarray*} 
% 
% A LA, PAR

Now, 
$$\int_{\bbbr}\limits \cos(\frac{\mu}{\sigma} \eta )\hat{\varphi}(\eta )d\eta = \varphi( \frac{\mu}{\sigma})=0$$
since $\mu>\sigma$. One may thus write:
\begin{eqnarray*}
  P_n^j & \geq & \frac{2}{\pi} \int_{0}^{+\infty}\limits \cos(\frac{\mu}{\sigma} \eta ) \hat{\varphi}(\eta) (e^{-\int_{t_j}^{t_{j+1}} |\frac{\eta}{\sigma}|^{\alpha(x)} dx} - 1) d\eta\\
 & =: & \frac{2}{\pi} n^{\alpha(t_j) \beta - \alpha(t_j) \varepsilon - 1} I_n^j,
 \end{eqnarray*} 
  
  where $I_n^j = I_{n,1}^j + I_{n,2}^j + I_{n,3}^j$ with 
  \begin{eqnarray*}
   I_{n,1}^j & = & \int_{0}^{+\infty}\limits \hat{\varphi}(\eta) \left(\cos(\frac{\mu}{\sigma} \eta ) - \cos(\eta) \right)n^{1-\alpha(t_j) \beta + \varepsilon \alpha(t_j)} \left(e^{-\int_{t_j}^{t_{j+1}} |\frac{\eta}{\sigma}|^{\alpha(x)} dx} - 1 \right) d\eta,\\
   I_{n,2}^j & = & \int_{0}^{+\infty}\limits \hat{\varphi}(\eta) \cos(\eta) n^{1-\alpha(t_j) \beta + \varepsilon \alpha(t_j)} \left(e^{-\int_{t_j}^{t_{j+1}} |\frac{\eta}{\sigma}|^{\alpha(x)} dx} - 1 + \int_{t_j}^{t_{j+1}} |\frac{\eta}{\sigma}|^{\alpha(x)} dx\right) d\eta,\\
   I_{n,3}^j & = & - \int_{0}^{+\infty}\limits \hat{\varphi}(\eta) \cos(\eta) n^{1-\alpha(t_j) \beta + \varepsilon \alpha(t_j)} \int_{t_j}^{t_{j+1}} |\frac{\eta}{\sigma}|^{\alpha(x)} dx d\eta.
  \end{eqnarray*}

  Let us show that there exists $K_U >0$ such that, for all $j \in J_n(U)$, 
  $$n^{1-\alpha(t_j) \beta + \varepsilon \alpha(t_j)} \int_{t_j}^{t_{j+1}} |\frac{\eta}{\sigma}|^{\alpha(x)} dx \leq K_U (\eta^c + \eta^d).$$
  Since $\frac{1}{\sigma} =  \frac{2 n^{\beta + \varepsilon}}{n^{2\varepsilon}-1}$ and, there exits $K$ such that, for all 
  $x \in (t_j,t_{j+1})$, $|\alpha(x) - \alpha(t_j)| \leq \frac{K}{n}$, one has 
  $|\sigma^{\alpha(t_j) - \alpha(x)}| \leq K$ and
  
  \begin{eqnarray*}
   n^{1-\alpha(t_j) \beta + \varepsilon \alpha(t_j)} \int_{t_j}^{t_{j+1}} |\frac{\eta}{\sigma}|^{\alpha(x)} dx & = & \frac{n^{-\alpha(t_j) \beta + \varepsilon \alpha(t_j)}}{\sigma^{\alpha(t_j)}} n \int_{t_j}^{t_{j+1}} |\eta|^{\alpha(x)} \sigma^{\alpha(t_j) - \alpha(x)} dx\\
   & \leq & 2^{\alpha(t_j)} \frac{n^{2 \varepsilon \alpha(t_j)}}{(n^{2 \varepsilon}-1)^{\alpha(t_j)}} (\eta^c + \eta^d) n \int_{t_j}^{t_{j+1}} \sigma^{\alpha(t_j) - \alpha(x)} dx\\
   & \leq & K (\eta^c + \eta^d).
  \end{eqnarray*}
As a consequence, 
\begin{eqnarray*}
 |I_{n,1}^j| & \leq & \int_{0}^{+\infty}\limits |\hat{\varphi}(\eta)| |\frac{\mu}{\sigma}-1|  n^{1-\alpha(t_j) \beta + \varepsilon \alpha(t_j)} \int_{t_j}^{t_{j+1}}  |\frac{\eta}{\sigma}|^{\alpha(x)} dx d\eta\\
 & \leq & K |\frac{\mu}{\sigma}-1| \int_{0}^{+\infty}\limits |\hat{\varphi}(\eta)| (\eta^c + \eta^d) d\eta.
\end{eqnarray*}

One finally obtains that $\sup_{j \in J_n(U)}\limits |I_{n,1}^j| \leq K | \frac{2}{n^{2 \varepsilon}-1}|$ 
and $\lim_{n \rightarrow +\infty}\limits ( \sup_{j \in J_n(U)}\limits |I_{n,1}^j|) =0.$

Let us now deal with $I_{n,2}^j$.

\begin{eqnarray*}
 |I_{n,2}^j| & \leq & \int_{0}^{+\infty}\limits |\hat{\varphi}(\eta)|n^{1-\alpha(t_j) \beta + \varepsilon \alpha(t_j)} \frac{1}{2} \left( \int_{t_j}^{t_{j+1}}  |\frac{\eta}{\sigma}|^{\alpha(x)} dx \right)^2 d\eta\\
 & \leq & \frac{K_U^2}{2} \int_{0}^{+\infty}\limits |\hat{\varphi}(\eta)|n^{1-\alpha(t_j) \beta + \varepsilon \alpha(t_j)} (\eta^c +\eta^d)^2 n^{2 \alpha(t_j) \beta - 2 - 2 \varepsilon \alpha(t_j)} d\eta\\
  & \leq & K_U n^{\alpha(t_j)(\beta - \varepsilon - \frac{1}{\alpha(t_j)} ) }\\
  & \leq & K_U n^{d_U(\beta - \varepsilon - \frac{1}{d_U }) }
 \end{eqnarray*}
and thus $\lim_{n \rightarrow +\infty}\limits ( \sup_{j \in J_n(U)}\limits |I_{n,2}^j|) =0.$

Fubini's theorem implies that 
\begin{eqnarray*}
 I_{n,3}^j & = & \int_{t_j}^{t_{j+1}} \frac{n^{1-\alpha(t_j) \beta + \varepsilon \alpha(t_j)}}{\sigma^{\alpha(x)}} \left( - \int_{0}^{+\infty }\limits \hat{\varphi}(\eta) \cos(\eta) \eta^{\alpha(x)} d\eta \right) dx\\
 & = &  \int_{t_j}^{t_{j+1}} \frac{n^{1-\alpha(t_j) \beta + \varepsilon \alpha(t_j)}}{\sigma^{\alpha(x)}} F(\alpha(x)) dx
 \end{eqnarray*}
 where $ F(\delta) = - \int_{0}^{+\infty }\limits \hat{\varphi}(\eta) \cos(\eta) \eta^{\delta} d\eta$.
 
The appendix contains a proof that $\min_{\delta \in [c,d]}\limits F(\delta) >0$. As a consequence, 
 
 \begin{eqnarray*}
 I_{n,3}^j & \geq & (\min_{\delta \in [c,d]}\limits F(\delta)) \frac{n^{-\alpha(t_j) \beta + \varepsilon \alpha(t_j)}}{\sigma^{\alpha(t_j)}} n \int_{t_j}^{t_{j+1}} \sigma^{\alpha(t_j)- \alpha(x)} dx\\
 & = & (\min_{\delta \in [c,d]}\limits F(\delta)) \left[ \frac{n^{2 \varepsilon}}{n^{2 \varepsilon}-1} \right]^{\alpha(t_j)} n \int_{t_j}^{t_{j+1}} e^{(\alpha(t_j)- \alpha(x)) \log \sigma} dx\\
 & \geq & K_U \left[ \frac{n^{2 \varepsilon}}{n^{2 \varepsilon}-1} \right]^{c} \left[ n \int_{t_j}^{t_{j+1}} (e^{(\alpha(t_j)- \alpha(x)) \log \sigma} -1)dx + 1\right].
 \end{eqnarray*}
 
Now $\lim_{n \rightarrow +\infty}\limits  \frac{n^{2 \varepsilon}}{n^{2 \varepsilon}-1}=1$ and 
$\lim_{n \rightarrow +\infty}\limits (\sup_{j \in J_n(U)}\limits |  n \int_{t_j}^{t_{j+1}} (e^{(\alpha(t_j)- \alpha(x)) \log \sigma} -1)dx|) =0$. As a consequence,
 
$$\exists K_U >0, \exists n_0 \in \bbbn: \forall n\geq n_0, \inf_{j \in J_n(U)}\limits  I_{n,3}^j \geq K_U >0$$ 

and thus

 $$  \inf_{j \in J_n(U)}\limits  (n^{1-\alpha(t_j) \beta + \varepsilon \alpha(t_j)}P_n^j) >0.$$
 \end{proof}

\begin{proof}[Proofs of Lemma \ref{LemE2maj} and Lemma \ref{LemE2min}]
 Parseval formula yields
 $$
 P_n^j = \frac{1}{\pi} \int_{\bbbr} \frac{\sin(\frac{\xi}{n^{\beta- \varepsilon}})- \sin(\frac{\xi}{n^{\beta + \varepsilon}})}{\xi} e^{- \int_{t_j}^{t_{j+1}} |\xi|^{\alpha(x)}dx} d\xi.$$
 
 Set $\mu_j = \frac{1}{\alpha(t_j)} + \varepsilon - \beta$ 
 with $\varepsilon \in (0, \beta - \frac{1}{c_U})$. Using the change of variable $\xi = n^{1/ \alpha(t_j)} v$, one gets
 
 $$
 P_n^j = \frac{1}{\pi} \int_{\bbbr} \frac{\sin(n^{\mu_j}v)- \sin(n^{\mu_j - 2 \varepsilon}v)}{v} \ e^{- \int_{t_j}^{t_{j+1}}\limits |v|^{\alpha(x)} n^{\frac{\alpha(x)}{\alpha(t_j)}}dx} dv.$$
 
 Define $$I_n^j =   \int_{\bbbr}e^{- \int_{t_j}^{t_{j+1}}\limits |v|^{\alpha(x)} n^{\frac{\alpha(x)}{\alpha(t_j)}}dx} dv.$$

 Since $\alpha$ is Lipschitz and $c\leq \alpha(x) \leq d$, one deduces that 
 $$\exists K^1_U >0, \exists K^2_U >0, \exists n_0 \in \bbbn, \quad  K_U^1 \leq \inf_{j \in J_n(U)}\limits I_n^j \leq \sup_{j \in J_n(U)}\limits I_n^j \leq K_U^2.$$
 
 Now,
 
 \begin{eqnarray*}
  P_n^j & = & n^{\mu_j} I_n^j \left( 1+ \frac{1}{I_n^j} \int_{\bbbr} ( \frac{\sin(n^{\mu_j}v)- \sin(n^{\mu_j - 2 \varepsilon}v)}{n^{\mu_j}v} -1) e^{- \int_{t_j}^{t_{j+1}}\limits |v|^{\alpha(x)} n^{\frac{\alpha(x)}{\alpha(t_j)}}dx} dv \right)\\
   & = & n^{\mu_j} I_n^j (1 + L_n^j).
 \end{eqnarray*}
 
 One computes: 
 
  \begin{eqnarray*}
   | L_n^j| & \leq & \frac{1}{K_U^1} \int_{\bbbr} \left[ | \frac{\sin(n^{\mu_j} v)}{n^{\mu_j} v}-1| + |\frac{\sin(n^{\mu_j-2\varepsilon} v)}{n^{\mu_j} v} | \right] e^{- \int_{t_j}^{t_{j+1}}\limits |v|^{\alpha(x)} n^{\frac{\alpha(x)}{\alpha(t_j)}}dx} dv\\
& \leq & \frac{1}{K_U^1} \int_{\bbbr} (\frac{1}{6} n^{2 \mu_j } v^2+ \frac{1}{n^{2\varepsilon}} ) e^{- \int_{t_j}^{t_{j+1}}\limits |v|^{\alpha(x)} n^{\frac{\alpha(x)}{\alpha(t_j)}}dx} dv\\
& \leq &  \frac{1}{6 K_U^1} n^{2(\frac{1}{c_U}+\varepsilon-\beta )}  \int_{\bbbr} v^2 e^{- \int_{t_j}^{t_{j+1}}\limits |v|^{\alpha(x)} n^{\frac{\alpha(x)}{\alpha(t_j)}}dx} dv + \frac{K_U^2}{K_U^2} \frac{1}{n^{2 \varepsilon}}.
 \end{eqnarray*}
There exist $K_U^3 \in \bbbr$ and $n_0 \in \bbbn$ such that, for all $n \geq n_0$, 
$$\sup_{j \in J_n(U)}\limits \int_{\bbbr} v^2 e^{- \int_{t_j}^{t_{j+1}}\limits |v|^{\alpha(x)} n^{\frac{\alpha(x)}{\alpha(t_j)}}dx} dv \leq K_U^3 <+\infty.$$ 
As a consequence, $\lim_{n \rightarrow +\infty}\limits (\sup_{j \in J_n(U)}\limits |L_n^j| ) =0.$

 \end{proof}

\subsection{Estimates for the number of increments and determination of the spectrum}
  
  \subsubsection{Lemmas}
  
  \begin{lem}\label{LemEncAcc1}
 Almost surely,  $\forall \beta < 0$, $f_g(\beta) = - \infty$.
  \end{lem}
  
   \begin{lem}\label{LemEncAcc2}
 Almost surely,  $\forall \beta \in (0, \frac{1}{d})$,  $f_g(\beta) = \beta d$.
  \end{lem}
  
     \begin{lem}\label{LemEncAcc3}
Almost surely, $f_g(0) = 0$.
  \end{lem}

     \begin{lem}\label{LemEncAcc4}
 Almost surely,  $\forall \beta \in [\frac{1}{d}, \frac{1}{c}]$, $f_g(\beta) = 1$.
  \end{lem}
  
%        \begin{lem}\label{LemEncAcc5}
% Almost surely, $f_g(\frac{1}{c}) = 1$.
%   \end{lem}
%   
       \begin{lem}\label{LemEncAcc6}
 Almost surely,  $\forall \beta \in (\frac{1}{c}, 1+\frac{1}{c})$, $f_g(\beta) = 1+\frac{1}{c} - \beta$.
  \end{lem}

       \begin{lem}\label{LemEncAcc7}
Almost surely, $f_g(1+\frac{1}{c}) = 0$.
  \end{lem}
  
         \begin{lem}\label{LemEncAcc8}
 Almost surely,  $\forall \beta > 1+\frac{1}{c}$, $f_g(\beta) = - \infty$.
  \end{lem}
  
    \subsubsection{Proofs}
  
  \begin{proof}[Proof of Lemma \ref{LemEncAcc1}]
   Fix $\beta <0$. Denote $E_{\beta} = (0,-\beta).$ If $\varepsilon \in E_{\beta}$, then
 \begin{eqnarray*}
  \P (\Nb \geq 1) & = & 1- \P (\Nb = 0 )\\
  & = & 1 - \prod_{j=1}^n\limits (1-P_n^j).
 \end{eqnarray*}
  Lemma \ref{LemE1maj} implies that, for $n$ large enough,
  \begin{eqnarray*}
  \P (\Nb \geq 1) & \leq &  1- (1-Kn^{d\beta + d\varepsilon-1})^n\\
  & \leq & 1-e^{-\frac{3}{2} K n^{d\beta + d\varepsilon}}\\
  & \leq & \frac{3}{2} K n^{d\beta + d\varepsilon},
  \end{eqnarray*} 
  and thus $\lim_{n \rightarrow +\infty}\limits \P (\Nb \geq 1) =0$. 
  Since $\Nb$ tends to $0$ in probability when $n$ tends to infinity, there exists a subsequence $\sigma(n)$ such that 
  $N_{\sigma(n)}^{\varepsilon}(\beta)$ tends to $0$ almost surely. This implies that,
  almost surely, $\liminf_{n \rightarrow +\infty}\limits \frac{\log \Nb}{\log n} =-\infty$.
  We have proved that:
    \begin{equation}\label{voidspectrum}
 \forall \beta < 0, \forall \varepsilon \in E_{\beta}, \mbox{almost surely, }
  \liminf_{n \rightarrow +\infty}\limits \frac{\log \Nb}{\log n} =-\infty.
  \end{equation}
  
Let $\Omega_{\beta, \varepsilon} = \{ \omega: (\ref{voidspectrum}) \textrm{ holds} \}$, $\Omega_{\beta} = \bigcap_{\varepsilon \in E_{\beta} \cap \mathbb{Q}}\limits  \Omega_{\beta, \varepsilon}$ and 
$\Omega =  \bigcap_{\beta \in (-\infty,0) \cap \mathbb{Q}}\limits  \Omega_{\beta}.$ Note that
$\Omega$ has probability $1$, and consider $\omega \in \Omega$.
Choose $\beta <0$ and $j \in \bbbn$. Set $\beta_j = \frac{[\beta j]}{j}$. For $j$ large enough,  
$N_{n}^{1/j}(\beta) \leq N_{n}^{2/j}(\beta_j)$. In addition 
$ \liminf_{n \rightarrow \infty}\limits \frac{\log  N_{n}^{2/j}(\beta_j)}{\log n} = -\infty$. As a consequence,
$ \liminf_{n \rightarrow \infty}\limits \frac{\log  N_{n}^{1/j}(\beta)}{\log n} = -\infty$ and, almost surely, for all $\beta<0$, $f_g(\beta) = -\infty.$
 
 \end{proof}

\begin{proof}[Proof of Lemma \ref{LemEncAcc2}]
     Fix $\beta \in (0,\frac{1}{d})$. Denote $E_{\beta} =\{ \varepsilon \in (0, \min(\beta, \frac{1}{d}-\beta)) \textrm{ such that } (d-\varepsilon)(\beta-\varepsilon) \in (0,1) \} $. Choose $\varepsilon \in E_{\beta}$. By Lemma \ref{LemE1maj},  
 there exists $K>0$  and $n_0 \in \bbbn$ such that, for all $n \geq n_0$ and all $j \in \llbracket 1,n \rrbracket$, 
   $$P_n^j \leq K n^{d \beta + d \varepsilon -1}.$$ Lemma \ref{LemPR7} then implies that, almost surely, 
   
   \begin{equation}\label{inegbeta0}
   \liminf_{n \rightarrow +\infty}\limits \frac{\log \Nb}{\log n} \leq d \beta + d \varepsilon.
   \end{equation}

   There exists an open interval $U$ such that, for all $t \in U$, $\alpha(t) \geq d - \varepsilon$. 
   Using Lemma \ref{LemE1min}, there exist $K>0$ and $n_0 \in \bbbn$ such that, for all $n \geq n_0$
   and all $j \in J_n(U)$, 
   $$P_n^j \geq K n^{(d- \varepsilon)( \beta - \varepsilon) -1},$$
   and Lemma \ref{LemPR8} then implies that, almost surely, 
   
   \begin{equation*}
   \liminf_{n \rightarrow +\infty}\limits \frac{\log \Nb}{\log n} \geq (d- \varepsilon)( \beta - \varepsilon).
   \end{equation*}
   We thus have proved that, for all $\beta \in (0,\frac{1}{d})$ and all $\varepsilon \in E_{\beta}$, almost surely, 
   \begin{equation}\label{inegalbetad}
    (d- \varepsilon)( \beta - \varepsilon) \leq  \liminf_{n \rightarrow +\infty}\limits \frac{\log \Nb}{\log n} \leq  d \beta + d \varepsilon.
   \end{equation}
 %which implies that $f_g(\beta) = d \beta$. 
 Then Lemma \ref{inversionencadrement} ensures that almost surely, for all $\beta \in (0,\frac{1}{d})$, $f_g(\beta) = d \beta.$

% Set $\Omega_{\beta, \varepsilon} = \{ \omega: (\ref{inegalbetad}) \textrm{ holds} \}$, $\Omega_{\beta} = \bigcap_{\varepsilon \in E_{\beta} \cap \mathbb{Q}}\limits  \Omega_{\beta, \varepsilon}$ and 
% $\Omega =  \bigcap_{\beta \in (0,\frac{1}{d}) \cap \mathbb{Q}}\limits  \Omega_{\beta}.$ 
% The event $\Omega$ has probability $1$. Choose $\omega \in \Omega$.
% Let $\beta \in (0,\frac{1}{d})$ and $j \in \bbbn$. Denote $\beta_j = \frac{[\beta j]}{j}$. For $j$ large enough,
% $N_{n}^{1/j}(\beta) \leq N_{n}^{2/j}(\beta_j)$. Also, 
% $ \liminf_{n \to \infty}\limits \frac{\log  N_{n}^{2/j}(\beta_j)}{\log n} \leq d \beta_j + \frac{2d}{j}$ and thus  
% $ \liminf_{n \to \infty}\limits \frac{\log  N_{n}^{1/j}(\beta)}{\log n} \leq d \beta_j + \frac{2d}{j}$.
% As a consequence, $f_g(\beta) \leq d \beta$. 
% 
% {\bf On devrait pouvoir se passer de ça : }
% 
% Ensuite, par semi-continuité supérieure, 
% \begin{eqnarray*}
%  f_g(\beta) & \geq & \limsup_{j \rightarrow +\infty}\limits f_g(\beta_j)\\
%  & = & \limsup_{j \rightarrow +\infty}\limits d \beta_j\\
%  & = & d \beta.
% \end{eqnarray*}

    \end{proof}

 \begin{proof}[Proof of Lemma \ref{LemEncAcc3}]
We obtain Inequality (\ref{inegbeta0}) for $\beta=0$ by applying Lemma \ref{LemE1maj} and Lemma \ref{LemPR7}. This yields that, almost surely,
$$\liminf_{n \rightarrow +\infty}\limits \frac{\log N_{n}^{\varepsilon}(0)}{\log n} \leq d \varepsilon.$$ 
As a consequence, $f_g(0) \leq 0$.     

Then apply Lemma \ref{inversionlimsup} and Lemma \ref{LemEncAcc2} to obtain
\begin{eqnarray*}
  f_g(0) & \geq & \limsup_{j \rightarrow +\infty}\limits f_g(\frac{1}{j})\\
  & = & \limsup_{j \rightarrow +\infty}\limits  \frac{d}{j}\\
  & = & 0.
\end{eqnarray*}

\end{proof}

   \begin{proof}[Proof of Lemma \ref{LemEncAcc4}]
    Let $\beta \in (\frac{1}{d}, \frac{1}{c})$ and $\varepsilon \in (0,1)$. 
    
    Choose an open interval $U$ such that $\beta > \frac{1}{c_U}$ and $ \beta < \frac{1}{d_U} + 2 \varepsilon$. 
    Lemma \ref{LemE2min} then ensures that there exist $K>0$ and $n_0 \in \bbbn$ such that,
    for all $n \geq n_0$ and all $j \in J_n(U)$, 
   \begin{eqnarray*}
    P_n^j & \geq & K n^{\frac{1}{\alpha(t_j)}+\varepsilon-\beta}\\
    & \geq & K n^{\frac{1}{d_U}+\varepsilon-\beta}\\
    & \geq &  K n^{-\varepsilon}.
   \end{eqnarray*}
   By Lemma \ref{LemPR8}, for all $\beta \in (\frac{1}{d}, \frac{1}{c})$, almost surely, 
   
   \begin{equation*}
   \liminf_{n \rightarrow +\infty}\limits \frac{\log \Nb}{\log n} \geq 1- \varepsilon.
   \end{equation*}
%which implies $f_g(\beta) \geq 1$ and finally $f_g(\beta) = 1$. 
We conclude by applying Lemma \ref{inversionencadrement}.

% On pose $\Omega_{\beta} = \{ \omega | f_g(\beta) = 1 \}$, $\Omega =  \bigcap_{\beta \in (\frac{1}{d}, \frac{1}{c}) \cap \mathbb{Q}}\limits  \Omega_{\beta}$, et $\beta_j = \frac{[j\beta]}{j}$.
% Soit $\beta \in (\frac{1}{d}, \frac{1}{c})$. Par semi-continuité supérieure, 
% \begin{eqnarray*}
%  f_g(\beta) & \geq & \limsup_{j \rightarrow +\infty}\limits f_g(\beta_j)\\
%  & = & \limsup_{j \rightarrow +\infty}\limits 1\\
%  & = & 1.
% \end{eqnarray*}
For $\beta = \frac{1}{c}$, we apply Lemma \ref{inversionlimsup} to obtain 
\begin{eqnarray*}
  f_g(\frac{1}{c}) & \geq & \limsup_{j \rightarrow +\infty}\limits f_g(\frac{[\beta j]}{j})\\
  & = & \limsup_{j \rightarrow +\infty}\limits 1.
\end{eqnarray*}

For $\beta = \frac{1}{d}$, Lemma \ref{inversionlimsup} and Lemma \ref{LemEncAcc2} lead to 
\begin{eqnarray*}
  f_g(\frac{1}{d}) & \geq & \limsup_{j \rightarrow +\infty}\limits f_g(\frac{[\beta j]}{j})\\
  & = & \limsup_{j \rightarrow +\infty}\limits d \frac{[\beta j]}{j}\\
  & = & 1.
\end{eqnarray*}
\end{proof}

 \begin{proof}[Proof of Lemma \ref{LemEncAcc6}]
      Let $\beta \in (\frac{1}{c}, 1+\frac{1}{c})$.  Denote $E_{\beta} = (0, \min(\beta- \frac{1}{c},1+\frac{1}{c} - \beta ) )$. Fix $\varepsilon \in E_{\beta}$. By Lemma \ref{LemE2maj}, there exist $K>0$ and $n_0 \in \bbbn$ such that,
      for all $n \geq n_0$ and all $j \in \llbracket 1,n \rrbracket$, 
   $$P_n^j \leq K n^{\frac{1}{c}+\varepsilon - \beta}.$$ 
   Lemma \ref{LemPR7} then implies that, almost surely, 
   
   \begin{equation}\label{inegbeta1}
   \liminf_{n \rightarrow +\infty}\limits \frac{\log \Nb}{\log n} \leq 1+\frac{1}{c}+\varepsilon - \beta.
   \end{equation}

   Choose an open interval $U$ such that $\frac{1}{c} - 2 \varepsilon \leq \frac{1}{d_U} \leq \beta - \varepsilon $. 
   By Lemma \ref{LemE2min}, there exist $K>0$  and $n_0 \in \bbbn$ such that, for all$ n \geq n_0$ and all
   $j \in J_n(U)$, 
   $$P_n^j \geq K n^{\frac{1}{d_U}+\varepsilon - \beta}.$$ 
   Lemma \ref{LemPR8} then implies that, almost surely, 
   
   \begin{equation*}
   \liminf_{n \rightarrow +\infty}\limits \frac{\log \Nb}{\log n} \geq 1+ \frac{1}{d_U}+\varepsilon - \beta \geq 1+ \frac{1}{c} - \varepsilon - \beta.
   \end{equation*}
   We have proved that, for all $\beta \in (\frac{1}{c}, 1+\frac{1}{c})$ and all $\varepsilon \in E_{\beta}$, 
   almost surely,
   \begin{equation}\label{inegalbetac}
    1+\frac{1}{c}-\varepsilon - \beta \leq  \liminf_{n \rightarrow +\infty}\limits \frac{\log \Nb}{\log n} \leq 1+\frac{1}{c}+\varepsilon - \beta.
   \end{equation}
%and thus $f_g(\beta) = 1+\frac{1}{c} - \beta$.

% Notons $\Omega_{\beta, \varepsilon} = \{ \omega | (\ref{inegalbetac}) \textrm{ holds} \}$, $\Omega_{\beta} = \bigcap_{\varepsilon \in E_{\beta} \cap \mathbb{Q}}\limits  \Omega_{\beta, \varepsilon}$ and 
% $\Omega =  \bigcap_{\beta \in (\frac{1}{c}, 1+ \frac{1}{c}) \cap \mathbb{Q}}\limits  \Omega_{\beta}.$ On considère l'événement $\Omega$ de probabilité $1$.
% Soit $\beta \in (\frac{1}{c}, 1+ \frac{1}{c})$ and $j \in \bbbn$. On note $\beta_j = \frac{[\beta j]}{j}$. For $j$ large enough, on a 
% $N_{n}^{1/j}(\beta) \leq N_{n}^{2/j}(\beta_j)$ et $ \liminf_{n \infty}\limits \frac{\log  N_{n}^{2/j}(\beta_j)}{\log n} \leq 1+\frac{1}{c}+\frac{2}{j} - \beta_j$ donc  
% $ \liminf_{n \infty}\limits \frac{\log  N_{n}^{1/j}(\beta)}{\log n} \leq 1+\frac{1}{c}+\frac{2}{j} - \beta_j$ et $f_g(\beta) \leq 1+\frac{1}{c} - \beta$. Ensuite, par semi-continuité supérieure, 
% \begin{eqnarray*}
%  f_g(\beta) & \geq & \limsup_{j \rightarrow +\infty}\limits f_g(\beta_j)\\
%  & = & \limsup_{j \rightarrow +\infty}\limits 1+\frac{1}{c} - \beta_j\\
%  & = & 1+\frac{1}{c} - \beta.
% \end{eqnarray*}

The result then follows from Lemma \ref{inversionencadrement}.
    \end{proof}

 \begin{proof}[Proof of Lemma \ref{LemEncAcc7}]

We obtain Inequality (\ref{inegbeta1}) for $\beta=1+\frac{1}{c}$ by applying Lemma \ref{LemE2maj} and Lemma \ref{LemPR7}. This yields that, almost surely,
$$\liminf_{n \rightarrow +\infty}\limits \frac{\log N_{n}^{\varepsilon}(1+\frac{1}{c})}{\log n} \leq d \varepsilon.$$ 
As a consequence, $f_g(1+\frac{1}{c}) \leq 0$.   

Then, Lemma \ref{inversionlimsup} and Lemma \ref{LemEncAcc6} lead to 
\begin{eqnarray*}
  f_g(1+\frac{1}{c}) & \geq & \limsup_{j \rightarrow +\infty}\limits f_g(\frac{[(1+\frac{1}{c}) j]}{j})\\
  & = & \limsup_{j \rightarrow +\infty}\limits (1+\frac{1}{c} - \frac{[(1+\frac{1}{c}) j]}{j})\\
  & = & 0.
\end{eqnarray*}

\end{proof}

  \begin{proof}[Proof of Lemma \ref{LemEncAcc8}]
   Let $\beta > 1+ \frac{1}{c}$ and denote $E_{\beta} = (0, \beta - 1 - \frac{1}{c})$. 
   For $\varepsilon \in E_{\beta}$, 
   \begin{eqnarray*}
  \P (\Nb \geq 1) & = & 1- \P (\Nb = 0 )\\
  & = & 1 - \prod_{j=1}^n\limits (1-P_n^j).
 \end{eqnarray*}
  Lemma \ref{LemE2maj} ensures that, for $n$ large enough,
  \begin{eqnarray*}
  \P (\Nb \geq 1) & \leq &  1- (1-Kn^{\frac{1}{c}+\varepsilon - \beta})^n\\
  & \leq & \frac{3}{2} K n^{1+\frac{1}{c}+\varepsilon - \beta}.
  \end{eqnarray*} 
Thus, $\lim_{n \rightarrow +\infty}\limits \P (\Nb \geq 1) =0$, which implies that
$\Nb$ tends to $0$ in probability when $n$ tends to infinity. There exists a subsequence $\sigma(n)$ such that 
  $N_{\sigma(n)}^{\varepsilon}(\beta)$ tends to $0$ almost surely.
    As a consequence, for all $\beta > 1 + \frac{1}{c}$ and all $\varepsilon \in E_{\beta}$, almost surely, 
  \begin{equation}\label{voidspectrum1}
   \liminf_{n \rightarrow +\infty}\limits \frac{\log \Nb}{\log n} =-\infty.
  \end{equation}
  We conclude as in the last part of the proof of Lemma \ref{LemEncAcc1}.
  \end{proof}

\section*{Appendix}\label{sec:annex}
The following result is due to R. Schelling \cite{RS}:
\begin{lem}
For all $\beta >0$, 
\begin{equation}\label{todo1}
 F(\beta) = -\int_0^\infty \eta^\beta \cos(\eta) \widehat \varphi(\eta) d\eta >0.
\end{equation}
\end{lem}

\begin{proof}
The L\'evy--Khintchine formula yields
$$
    |\eta|^\beta = \int_{y\neq 0} (1-\cos (y\eta))\nu_\beta(dy)
    \quad\text{with}\quad
    \nu_\beta(dy) = \frac{c_\beta\,dy}{|y|^{1+\beta}}.
$$

By Fubini's theorem,
\begin{align*}
    \int_0^\infty &\eta^\beta \cos(\eta) \widehat\varphi(\eta)\,d\eta\\
    &= \int_{y\neq 0} \int_0^\infty  (\cos(\eta) - \cos(y\eta) \cos(\eta)) \widehat\varphi(\eta)\,d\eta\,\nu_\beta(dy)\\
    &= \int_{y\neq 0} \int_0^\infty  \left(\cos(\eta) - \tfrac 12 \cos (1+y)\eta - \tfrac 12\cos(1-y)\eta\right) \widehat\varphi(\eta)\,d\eta\,\nu_\beta(dy) \\
    & = c \int_{y\neq 0} \left(\varphi(1) - \tfrac 12 \varphi(1+y) - \tfrac 12\varphi(1-y)\right) \nu_\beta(dy),
\end{align*}
where $c$ is a positive constant.
By definition, $\varphi$ is smooth and supported on$[-1,1]$, thus $\varphi(1)=0$. As a consequence, we find that 
$$\int_0^\infty \eta^\beta \cos(\eta) \widehat\varphi(\eta)\,d\eta = - \frac{c}{2} \int_{y\neq 0} \left(\varphi(1+y) + \varphi(1-y)\right) \nu_\beta(dy)<0.$$ 
This is inequality \eqref{todo1}.
\end{proof}
It is easy to see that the function $\beta \mapsto F(\beta)$ is continuous. As a consequence, 
$\min_{\delta \in [c,d]} F(\delta) >0$.

\end{document}